\documentclass[11pt,a4paper,reqno]{amsart}
\usepackage{amsaddr}
\usepackage{amsmath,amsfonts,amssymb,amsthm,mathrsfs}
\usepackage[utf8]{inputenc}
\usepackage[T1]{fontenc}
\usepackage{lmodern}
\usepackage{latexsym}
\usepackage{cancel}
\usepackage{microtype}
\usepackage{caption}
\usepackage{MnSymbol}
\usepackage{xcolor}
\usepackage{enumerate}
\usepackage[normalem]{ulem}
\usepackage{bbm}
\usepackage{geometry}
\usepackage{marginnote}
\usepackage{mathrsfs}

\usepackage[pagebackref,colorlinks=true,linkcolor=blue,citecolor=magenta]{hyperref}

\usepackage{graphics,tikz}

\newcommand{\BBr}{{\mathscr B}}
\newcommand{\MMr}{{\mathscr M}}
\newcommand{\RRr}{{\mathscr R}}
\newcommand{\SSr}{{\mathscr S}}

\newcommand{\fch}{{\mathbf{1}}}

\newcommand{\BB}{{\mathcal  B}}
\newcommand{\DD}{{\mathcal  D}}
\newcommand{\EE}{{\mathcal  E}}
\newcommand{\FF}{{\mathcal  F}}
\newcommand{\GG}{{\mathcal  G}}

\newcommand{\TT}{{\mathcal  T}}
\newcommand{\KK}{{\mathcal  K}}

\newcommand{\OO}{{\mathcal  O}}
\newcommand{\PP}{{\mathcal  P}}

\newcommand{\BE}{{\mathbb E}}

\newcommand{\BM}{{\mathbb M}}

\newcommand{\BR}{{\mathbb R}}

\def\XXint#1#2#3{{\setbox0=\hbox{$#1{#2#3}{\int}$ }
\vcenter{\hbox{$#2#3$ }}\kern-.6\wd0}}

\newtheorem{theorem}{\bf Theorem}[section]
\newtheorem{proposition}[theorem]{\bf Proposition}
\newtheorem{lemma}[theorem]{\bf Lemma}%[subsection]
\newtheorem{corollary}[theorem]{\bf Corollary}

\theoremstyle{definition}
\newtheorem{definition}[theorem]{Definition}
\newtheorem{example}[theorem]{\bf Example}

\newtheorem{remark}[theorem]{Remark}

\numberwithin{equation}{section}

\begin{document}

\title[Poisson equation with measure data]{Poisson equation
with measure data, reconstruction formula and Doob classes of processes}

\maketitle
\begin{center}

\normalsize
TOMASZ KLIMSIAK\textsuperscript{1,2}
ANDRZEJ ROZKOSZ\textsuperscript{2}
\par \bigskip

\textsuperscript{1} {\small Institute of Mathematics, Polish Academy of Sciences,\\
\'{S}niadeckich 8,   00-656 Warsaw, Poland} \par \medskip

 \textsuperscript{2} {\small Faculty of
Mathematics and Computer Science, Nicolaus Copernicus University,\\
Chopina 12/18, 87-100 Toru\'n, Poland }\par

\footnote{e-mail: {\tt tomas@mat.umk.pl} (T. Klimsiak), {\tt rozkosz@mat.umk.pl} (A. Rozkosz)}
\end{center}

\begin{abstract}
We consider the Dirichlet problem for equation
involving a general operator associated with a symmetric transient regular Dirichlet form and
bounded Borel measure on the right-hand side of the equation.
We introduce a  new function space (depending on the form) which allows us to distinguish between solutions with diffuse measure and with general Borel measure.
This  new space  can by characterized analytically in terms of the Poisson kernel associated with the underlying  operator or probabilistically by using the notion of Doob class (D) of processes naturally associated with the operator. We also prove a reconstruction formula describing, in terms of the carr\'e du champ operator
and jump measure associated with the underlying form,  the behaviour of the solution on the set where it is very large.
\end{abstract}

{\sc MSC-Classification}: 35R06, 35R11, 60H30.

\maketitle

\section{Introduction}
\label{sec1}

Let $D\subset\BR^d$, $d\ge2$, be a bounded regular domain and $\mu$ belongs to the set $\MMr_b(D)$ of all signed Borel measures on $D$ having finite total variation. It is well known that then there exists a unique weak solution $u$  to the Dirichlet problem
\begin{equation}
\label{eq1.1}
-\Delta u=\mu\quad\mbox{in }D,\quad u|_{\partial D}=0,
\end{equation}
i.e. $u\in L^1(D)$ and  for any $\eta\in \mathcal C:=\{u\in C^2(\bar D): u=0 \text{ on } \partial D\}$ we have
\[
-\int_Du\Delta\eta=\int_D\eta\,d\mu.
\]
In fact, $u$ is given by the formula
\begin{equation}
\label{eq1.2}
u(x)=\int_DG_D(x,y)\,\mu(dy),\quad x\in D,
\end{equation}
where $G_D$ is the Green kernel of the operator $-\Delta$ in $D$
(see, e.g., \cite{MV} and Section \ref{sec7} for more details). It is also known that  $u$ defined by  (\ref{eq1.2}) is quasi-continuous
(with respect to the Newtonian capacity) and  $u\in W^{1,q}_0(D)$ with $q\in[1,d/(d-1))$.
Suppose now that $\mu$ belongs to the set $\MMr_{0,b}(D)$ of diffuse measures,
i.e. $\mu\in\MMr_b(D)$ and $\mu$ charges no set of Newtonian capacity zero,
for instance $\mu(dx)=f(x)\,dx$ with $f\in L^1(D)$ or $\mu\ll \mathcal H^{\beta}$, where $\mathcal H^{\beta}$
is the $\beta$-dimensional Hausdorff measure on $D$ with $\beta\in (d-2,d)$.
Intuitively, in that case $u$ should have better regularity properties than in the case of general bounded Borel measure.
The question  arises whether this is true and one can find some regularity property   which allows one to distinguish between solutions of (\ref{eq1.1}) with $\mu\in\MMr_b(D)$ and  with $\mu\in\MMr_{0,b}(D)$. Known to us estimates for $u$,  like the aforementioned estimate in the Sobolev space $W^{1,q}_0(D)$,  depend only on the total variation of $\mu$.
%so the question is rather delicate.
This suggests that the solution of the problem should make use of some finer properties of $u$ defined by (\ref{eq1.2}).
In the present paper, we treat  this problem by using results from the probabilistic potential theory.

Let $\mathcal O$ denote the family of all open subsets of $D$. For   $V\in\mathcal O$ we denote by  $(P_V(x,dy))_{x\in V}$  the Poisson kernel of $-\Delta$, i.e.
the kernel such that  $P_V(x,dy)$ is a probability measure on $\partial V$ for any $x\in V$ and   for any $\varphi\in C_b(\partial D)$ the function
\[
P_V\varphi(x)= \int_{\partial V}\varphi(y)\, P_V(x,dy), \quad x\in D,
\]
is the unique solution to the Dirichlet problem
\[
-\Delta u=0\quad \text{in}\quad V,\qquad u=\varphi \quad\text{on}\quad \partial V.
\]
 Let FVP  denote the set of all increasing convex functions $\varphi:\BR^+\rightarrow\BR^+$
 with $\varphi(0)=0$ such that $\lim_{x\rightarrow\infty}\varphi(x)/x=\infty$. Elements of FVP
 will be called de la Vall\'ee--Poussin functions. We introduce the space   $\DD^{1,c}$ (see Theorem \ref{th4.5}) that consists of
 quasi-continuous functions  $u\in L^1(D)$  satisfying
\begin{equation}
\label{eq1.3}
\sup_{V\in\mathcal O} \|P_V\varphi(|u|)\|_{L^1(D)}<\infty\quad\mbox{for some }\varphi\in\mbox{FVP}.
\end{equation}
It appears (Proposition \ref{prop4.11}, Theorem \ref{th4.5}) that $\DD^{1,c}$ with the norm
\[
\|u\|_{\DD^1(D)}= \sup_{V\in\mathcal O}\int_D P_V(|u|)(x)\,dx
\]
is a separable Banach space. Our first main result (Theorem \ref{th4.5}, Theorem \ref{th5.8}) asserts that
\begin{equation}
\label{state.main}
\mu\quad \mbox{is diffuse if and only if}\quad u\in\DD^{1,c}(D).
\end{equation}
In order to present our second main result, let us recall a reconstruction formula
that follows from \cite[Theorem 2.33]{DMOP}:
\begin{equation}
\label{eq1.7}
\lim_{n\rightarrow\infty}\frac{1}{n}\int_{\{n\le u\le 2n\}}\eta\nabla u\cdot\nabla u\,dx=\int_D\eta\,d\mu_c^+,\quad \eta\in C_c(D),
\end{equation}
where $\mu_c$ is the concentrated part of $\mu$ (i.e. $\mu_c$ is orthogonal to the Newtonian capacity).
It provides an  information on $\mu_c$  based on the behaviour of the energy of $u$ on the set where $u$ is very large. The second goal of the paper is to establish a form of the reconstruction formula that
is somehow compatible  with the space $\DD^{1,c}(D)$ and at the same time  is suitable for generalizations  to the wide class of self-adjoint operators that generate  Markov semigroups (e.g.  fractional Laplacian). We show the following reconstruction formula (Theorem \ref{th5.8}): for any weight $\rho$ on $D$, i.e. a strictly positive function with $\|\rho\|_{L^1(D)}=1$,  we have
\begin{equation}
\label{eq1.4}
\lim_{n\rightarrow\infty}\sup_{V\in\mathcal O}\int_D \rho(x) P_V(|u|-n)^+(x)\,dx
=\int_D G_D\rho(x)\,|\mu_c|(dx).
\end{equation}
The above formula may be read as follows:
(a) If $\mu_c$ is non-zero, then for any $n\ge 1$ and  weight $\rho$ on $D$
there exists an open set $V\subset D$ and a  harmonic function $h^n_V$ on $V$ with
the prescribed boundary data
$(|u|-n)^+$ such that $\|h^n_V\|_{L^1_\rho(V)}$ is close to $\|G_D|\mu_c|\|_{L^1_\rho(D)}$, (b)
On the other hand, if $\mu_c=0$,  then for any  weight $\rho$,
 $\sup_{V\in\mathcal O}\|h^n_V\|_{L^1_\rho(V)}$ is eventually small.

For a better understanding of the space $\mathcal D^{1,c}(D)$, we provide
a couple of results that   characterize it.
First, we observe that \eqref{eq1.4} together with \eqref{state.main} imply that  for a quasi-continuous function $u$ on $D$ we have
 \begin{equation}
 \label{eq.equiv1}
 u\in\mathcal D^{1,c}(D)\quad \Leftrightarrow \quad \lim_{n\rightarrow\infty}\sup_{V\in\mathcal O} \|P_V(|u|-n)^+\|_{L^1(D)}=0.
 \end{equation}
In fact, in the paper we adopt  \eqref{eq.equiv1} as the  basic definition of $\DD^{1,c}(D)$
(see the beginning of Section \ref{sec4})  and then  we prove in Theorem \ref{th4.5}
that \eqref{eq1.3} is an equivalent formulation.  Interestingly,   in the definition of the norm $\|\cdot\|_{\mathcal D^1}$  the supremum
sign can be moved, preserving equality,   under the integral sign (Proposition  \ref{prop4.1}), i.e. for any  $u\in \DD^{1,c}(D)$,
\[
\|u\|_{\DD^1(D)}=\|\sup_{V\in\OO}P_V(|u|)\|_{L^1(D)},
\]
which implies that for a quasi-continuous function $u$ on $D$,
 \begin{equation}
 \label{eq.equiv12}
 u\in\mathcal D^{1,c}(D)\quad \Leftrightarrow \quad \lim_{n\rightarrow\infty} \|\sup_{V\in\mathcal O}P_V(|u|-n)^+\|_{L^1(D)}=0.
 \end{equation}
The function $e_{|u|}:=\sup_{V\in\OO}P_V(|u|)$ may be regarded as a generalized solution
to the obstacle problem for \eqref{eq1.1} with the barrier $|u|$ (its measurability (nearly Borel) is a consequence of \cite[Theorem V.1.17]{BG}).
This puts a different perspective on the space  $\DD^{1,c}(D)$.  It permits us to view it  as the class of quasi continuous functions
$u\in L^1(D)$ with the property that the $L^1$-norm of generalized solutions of the obstacle problem for \eqref{eq1.1}
with barriers $(|u|-n)^+$ tends to zero when $n\to \infty$.

Our third main result concerns  a probabilistic characterization of the space  $\DD^{1,c}(D)$
(Theorem \ref{th4.3}). This is a crucial point of the paper since our basic tools come from the probabilistic potential theory.
Let $(B_t)$ be a standard Brownian motion on a probability space $(\Omega,\FF,P)$.
Recall that the process $t\mapsto u(x+B_t)$ is called to be  of Doob class (D)   if the family of random variables
\[
\{u(x+B_{\tau}):\, \tau \text{ is a stopping time and } \tau\le\tau^x_D\},
\]
where $\tau^x_D=\inf\{t>0: x+B_t\notin D\}$,
is uniformly integrable under $P$.
The aforementioned characterization  is the following:
\[
u\in\DD^{1,c}(D)\quad\Leftrightarrow\quad t\mapsto u(x+B_t)\mbox{ is continuous of class (D)}\mbox{ for q.e. }x\in D
\]
(Here q.e. means {\em quasi everywhere}, i.e. except of a set  of the Newtonian capacity zero).
It explains and justifies the phrase ``Doob classes of processes"  in the title of the paper.

Above we have presented our main results for the Laplace operator, but in
 fact we prove them in a much more general setting. We consider the problem
\begin{equation}
\label{eq1.5}
-Lu=\mu\quad\mbox{in }D,\qquad u=0\quad\mbox{on }D^c:=E\setminus D,
\end{equation}
where $L$ is the operator corresponding to a symmetric,  transient and regular Dirichlet form
$(\EE,\mathfrak D(\EE))$ on $L^2(E;m)$ satisfying the absolute continuity condition (see Section \ref{sec2}). For instance, as $L$ we can take
a divergence form operator
\begin{equation}
\label{eq1.2qqq}
L=\sum^d_{i,j=1}\partial_{x_i}(a_{ij}(x)\partial_{x_j}),
\end{equation}
where the coefficients $a_{ij}\in\mathcal B(D)$ are bounded, the matrix $a:=[a_{ij}]$
is nonnegative definite a.e., and
$a$ is a.e. invertible  with $a^{-1}\in L^1_{loc}(D)$.
A fundamental class  of purely  nonlocal operators  consists of  L\'evy operators
\begin{equation}
\label{eq1.3qqq}
Lu(x)=  \text{p.v.}\int_{\mathbb R^d}(u(x+y)-u(x))\,\nu(dy),
\end{equation}
where $\nu$ is a symmetric  L\'evy measure: $\nu(dx)=\nu(-dx)$ and $\int_{\mathbb R^d}\min{\{1,|y|^2\}}\,\nu(dy)<\infty$.
In this case the absolute continuity condition holds provided that
the Hartman--Wintner condition holds for the Fourier symbol $\psi$ of $L$:
\[
\frac{\psi(\xi)}{\log(1+|\xi|)}\to \infty,\quad |\xi|\to \infty.
\]
Under the absolute continuity condition the resolvent associated with $L$ is determined by a density $r^D(x,y)$ ($r^D=G_D$ in case $L=\Delta$)
and one can define a  solution $u$ of (\ref{eq1.5}) by putting
\[
u(x)=R^D\mu(x),\quad\mbox{where}\quad R^D\mu(x)=\int_Dr^D(x,y)\,\mu(dy)
\]
for quasi-every (q.e.) $x\in D$ and $u=0$ on $D^c$. Roughly speaking, we prove that
then the main results stated above  for equation (\ref{eq1.1})
also hold for (\ref{eq1.5}) but with the Poisson kernel of $-\Delta$ replaced by the Poisson kernel of $-L$, $G_D$
replaced by $r^D$, Brownian motion $(B_t)$ replaced by a suitable Hunt process $(X,P_x)$ corresponding to
$L$ and, in general, the family $\mathcal O$ replaced by the broader family  $\mathcal O_q$ of all quasi open subsets
of $D$ (although we prove that if $R^D_1(C_b(D))\subset C_b(D)$, then as in the case of Laplace operator,
in the formulation of the results we can restrict ourselves to the family $\OO$; see Proposition \ref{prop4.8}).
The advantage of the potential theory approach to \eqref{eq1.5} is that we can treat in a concise way
a wide class of operators of different character. Note that in our setting   the   Poisson measures are  in general  supported by $E\setminus D$.

Our last main result is a generalization of the reconstruction formula \eqref{eq1.7} to the
class of operators  considered in the paper. Recall that the energy form $\EE$ may be represented as the integral
of the so-called {\em carr\'e du champ operator} $\Gamma$:
\[
\EE(u,u)=\frac12\int_{E} d\Gamma(u,u),\quad u\in\mathfrak D(\EE).
\]
The operator $\Gamma$ can be decomposed into the local part $\Gamma_c$  and the jump part $\Gamma_j$
(see Section \ref{sec6}). The jump part admits the form
\[
\Gamma_j(u,u)(dx)= 2\int_{\mathbb R^d} |u(x)-u(y)|^2\,J(dx,dy),
\]
where $J$ is the jump measure
coming from the Beurling--Deny decomposition of the form $\EE$.
Both operators $\Gamma_c$ and $\Gamma_j$ are well defined on the domain $\mathfrak D(\EE)$.
We show (Theorem \ref{th6.2}) that $\Gamma_c$ may be naturally  extended to solutions of \eqref{eq1.5}, and for any $\eta\in C_c(D)$ we have
 \begin{align}
\label{eq1.12}
\frac1{2n}\Big[\int_{\{n\le u\le 2n\}}\eta\,d\Gamma_c(u,u)&+ \int_D\int_D \eta(x)\theta_n(u(x),u(y))\,J(dx,dy)\nonumber \\
&+\int_D \eta(x)\theta_n(u(x),0)\,\kappa_D(dx)\Big]\to \int_D \eta\,d\mu^+_c,
\end{align}
where $\kappa_D$ is the killing measure of the form $\EE$ restricted to $D$ and
\[
\theta_n(u(x),u(y))= 2\big(S_n(u(x))-S_n(u(y))\big)\big(2u(x)-S_n(u(x))-S_n(u(y))\big)
\]
with $S_n(z)= \max\{\min\{z,2n\},n\}$.
Comparing this result to the previously known
reconstruction formulas (studied before only for local operators) one can notice   some similarities as well as  some essential differences.
The local part of the carr\'e du champ operator $\Gamma$, in analogy with the case of local operators,  is restricted to the set where $u$ lies  between $n$ and $2n$,  but the non-local part of $\Gamma$ is modified in a much more subtle way.
Recall that $J(dx,\,dy)$ describes, roughly speaking, the intensity of  jumps from a region $dx$ to $dy$  of a particle that  moves according to  the dynamic given by $L$.
Formula \eqref{eq1.12} says  that  all the jumps  within  the set
$\{n\le u\le 2n\}$ contribute to $\mu_c^+$ and no jumps which occur only within  $\{u\le n\}$ or $\{u\ge 2n\}$ contribute to $\mu^+_c$.
Interestingly,  when one of the following cases happens: there are jumps
from $\{u\le n\}$ to $\{n\le u\le 2n\}$ or $\{u\le n\}$ to $\{u\ge 2n\}$ or $\{n\le u\le 2n\}$ to $\{u\ge 2n\}$, then the energy on the left-hand side of (\ref{eq1.12})
is suitably reduced but not to zero.

Our main results are proved in Sections \ref{sec4}--\ref{sec6}.
In Section \ref{sec7}, we provide some examples of forms satisfying the assumptions of the main theorems and  give remarks on the equivalence of (integral) solutions of (\ref{eq1.5}) to other notions of solutions (weak, duality and renormalized).

\section{Preliminaries}
\label{sec2}

In the paper, $E$ is a locally compact separable metric space and $D$  is an open nonempty subset  of $E$. We denote by $\partial$ a one-point compactification of $E$. If $E$ is already compact, then we adjoin $\partial$ to $E$ as an isolated point.
We denote by $\BBr(E)$ the $\sigma$-field of Borel subsets of $E$ and for  $B\in\BBr(E)$ we set
$\BBr(B):=\{W\in \BBr(E): W\subset B\}$. $\BB_b(B)$ is the set of all bounded Borel measurable functions on $B$.
We adopt the convention that every function $f$ on $D$ is extended to $D\cup\partial$ by setting $f(\partial)=0$.

\subsection{Dirichlet forms.}
Throughout what follows, we shall use some notions and results from the theory of Dirichlet forms and Markov processes. Most of them are found in the books \cite{CF,FOT}. For the convenience of the reader and to fix notation,  we recall here some of them.

Let $m$ be a nonnegative Radon measure on $E$ with full support, that is $m$ is a nonnegative Borel measure on $E$ finite on compact sets and  strictly positive on open sets.
In what follows $(\EE,\mathfrak D(\EE))$ is  a symmetric regular Dirichlet form on $L^2(E;m)$.
In the whole paper we assume that it is regular and transient.

By \cite[Theorem 2.1.3]{FOT}, each function $u\in \mathfrak D(\EE)$ admits a quasi-continuous $m$-version
that we denote by $\tilde u$.
We denote by $(\EE^D, \mathfrak D(\EE^D))$ the part of $(\EE,\mathfrak D(\EE))$ on $D$.
Recall that
\[
\mathfrak D(\EE^D)=\{u\in \mathfrak D(\EE):\tilde u\mbox{ q.e. on }D^c:=E\setminus D\},
\quad \EE^D(u,v)=\EE(u,v),\quad u,v\in \mathfrak D(\EE^D).
\]
By \cite[Theorem 4.4.3]{FOT},  $(\EE^D,\mathfrak D(\EE^D))$ is a symmetric regular Dirichlet form on $L^2(D;m)$.
We denote by  $\mathfrak D_e(\EE)$ the  extended Dirichlet space of $(\EE,\mathfrak D(\EE))$.
To simplify notation, we continue to write $F$ for $\mathfrak D_e(\EE)$.
Note that $F$ with the inner product $\EE$ is a Hilbert space. The dual space of $F$ is denoted by $F^*$.

Let $L$ be the operator associated with $(\EE,\mathfrak D(\EE))$, i.e. the unique nonpositive definite self-adjoint operator on $L^2(E;m)$ such that
\[
\mathfrak D(L)\subset \mathfrak D(\EE),\qquad \EE(u,v)=(-Lu,v),\quad u\in \mathfrak D(L),\,v\in \mathfrak D(\EE),
\]
where $(\cdot,\cdot)$ denotes the usual inner product in $L^2(E;m)$ (see \cite[Corollary 3.1]{FOT} for more details). The operator $L_D$ associated (in the above sense) with $(\EE^D,\mathfrak D(\EE^D))$ will be denoted by $L_D$.

We define  quasi notions (capacity, exceptional sets, quasi-continuous functions, etc.) with respect to $\EE$ (or $\EE^D$)  as in \cite{FOT}.
We say that a property of points in $E$ holds quasi-everywhere in $E$ (q.e.  in $E$ in abbreviation) if it holds outside some $\EE$-exceptional subset of $E$.
The capacity with respect to $\EE$ (resp. $\EE^D$) will be denoted by  $\mbox{Cap}$ (resp. $\mbox{Cap}_D$).

\subsection{Markov processes.}
Let $\Omega$ be a set of functions $\omega: [0,\infty)\to E\cup\partial$
that are right continuous and have left limits (c\`adl\`ag functions) and satisfy the following property:
if $\omega(t)=\partial$, then $\omega(s)=\partial,\, s\le t$.  We endow $\Omega$ with the Skorokhod metric (see \cite{Bil}). We let
\[
X_t(\omega):=\omega(t),\quad t\ge 0,\, \omega\in\Omega.
\]
By \cite[Theorems 4.2.8, 7.2.1]{FOT},  there exists a unique (up to equivalence) $m$-symmetric Hunt process
$\BM=((\FF_t)_{t\ge0},(P_x)_{x\in E\cup\partial})$ with state space $E$
%,life time $\zeta$ and cemetery state $\partial$,
associated with $(\EE,\mathfrak D(\EE))$. Here $(P_x)$ is a family of Borel probability measures on $\Omega$
and $\FF_t$ is a $\sigma$-algebra that is a suitable completion of
\[
\FF^0_t:=\sigma(X_s,\, s\le t).
\]
We denote by  $\BE_x$  the expectation with respect to the measure $P_x$.
Let $\PP$ denote the set of all probability measures on $\BBr(E)$ and let $\FF^0_{\infty}=\sigma(X_t,t\ge0)$.  For $\nu\in\PP$ we set
\[
P_\nu(\Lambda)= \int_EP_x(\Lambda)\,\mu(dx),\quad \Lambda\in\FF^0_{\infty}.
\]
The expectation with respect to $P_{\mu}$ will be denoted by  $\mathbb E_{\mu}$.

Let $\BBr^n(E)$ denote the family of all {\em nearly Borel measurable}
subsets of $E$ (see \cite[p. 392]{FOT}). For  $V\in \BBr^n(E)$ we let
\[
\tau_V=\inf\{t>0: X_t\notin V\}.
\]
By \cite[p. 392]{FOT},  $\tau_V$ is a stopping time. Note that from \cite[Theorem A.2.6, Theorem 4.1.3]{FOT} it follows that
\begin{equation}
\label{eq2.1}
P_x(\tau_V=0)=1 \quad \mbox{q.e. }x\in V^c.
\end{equation}
As it is customary, we let $\zeta:=\tau_E$.
We  denote by $\BM^D=((\FF_t)_{t\ge0},(P^D_x)_{x\in D\cup\partial})$ a Hunt process, called the part of $\BM$ on $D$,
associated with $(\EE^D,\mathfrak D(\EE^D))$. It is known   (see \cite[Section 4.4]{FOT}) that
%\[
%P_x^D(\zeta=\tau_D)=1,\quad x\in D,
%\]
%and
$P^D_x=y^D_{\sharp} P_x$, where  $y^D_{\sharp} P_x$ denotes the push-forward of the
measure $P_x$ through the mapping $y^D$ defined by
\[
y^D:\Omega\to\Omega,\quad y^D(\omega)(t):=\omega(t),\, t<\tau_D(\omega),\quad y^D(\omega)(t):=\partial,\, t\ge\tau_D(\omega).
\]
We denote by $(P^D_t)_{t>0}$ and $(R^D_{\alpha})_{\alpha>0}$ the transition semigroup and the resolvent of $\BM^D$, that is
\[
P^D_tf(x)=\mathbb E^D_xf(X_t)=\mathbb E_x[\fch_{\{t<\tau_D\}}f(X_t)],\qquad R^D_{\alpha}f(x)=\mathbb E_x\int^{\tau_D}_0e^{-\alpha t}f(X_t)\,dt,\quad x\in D,
\]
for any  $f\in\BB_b(D)$.
We also set $P_t(x,B)=P_t\fch_B(x)$, $R_{\alpha}(x,B)=R_{\alpha}\fch_B(x)$, $B\in\BBr(E)$.

In the paper we will assume that $\BM^D$ satisfies the absolute continuity condition, that is
\[
R^D_{\alpha}(x,\cdot)\ll m\quad \mbox{for any $\alpha>0$ and $x\in D$}.
\]
Equivalently (see \cite[Theorem 4.2.4]{FOT}),
\begin{equation}
\label{eq2.6}
P^D_t(x,\cdot)\ll m\quad \mbox{for any $t>0$ and $x\in D$}.
\end{equation}
By \cite[Lemma 4.2.4]{FOT},
if $\BM^D$ satisfies the  absolute continuity condition, then for every $\alpha>0$  there exists a nonnegative $\BBr(D)\otimes\BBr(D)$-measurable function $r^D_{\alpha}:D\times D\rightarrow\BR$
such that
\[
R^D_{\alpha}f(x)=\int_Dr^D_{\alpha}(x,y)f(y)\,m(dy),\quad x\in E,\,f\in\BB_b(D).
\]
Furthermore, there exists a nonnegative symmetric  $\BBr(E)\otimes\BBr(E)$-measurable function $r^D:E\times E\rightarrow\BR$ such that
\[
R^Df(x):=R^D_0f(x)=\int_Dr^D(x,y)f(y)\,m(dy),\quad x\in D,\,f\in\BB_b(D).
\]
In fact, $r^D(x,y)=\lim_{\alpha\downarrow0}r^D_{\alpha}(x,y)$ (see the remarks in \cite[p. 256]{BG}). We call $r^D$ the resolvent density. Note that for each $y\in D$, $r^D(\cdot,y)$  is an excessive function relative to  $(P^D_t)_{t>0}$. Recall that a positive nearly
Borel function $u$ is called  $(P^D_t)$-excessive whenever $P^D_tu(x)\le u(x)$, $t\ge  0$,
$x\in D$ and $\lim_{t\to 0^+}P^D_tu(x)=u(x),\, x\in D$.

In what follows for a nonnegative Borel measure $\mu$ on $E$ we set
\begin{equation}
\label{eq.not1}
R^D_{\alpha}\mu(x)=\int_Dr^D_{\alpha}(x,y)\,\mu(dy),\quad
R^D\mu(x)=\int_Dr^D(x,y)\,\mu(dy),\quad x\in E.
\end{equation}

\subsection{Concentrated and smooth measures.} We denote  by  $\SSr(E)$ the set of all {\em smooth measures} on $E$.
Recall that a nonnegative measure $\mu$ belongs to $\SSr(E)$ if  there exists
an increasing sequence $\{F_n\}$ of closed subsets of $E$ such that
$\mbox{Cap}(K\setminus F_n)\to 0$ as $n\rightarrow\infty$ for every compact $K\subset E$
and $\mathbf1_{F_n}\cdot\mu\in F^*$, $n\ge1$ (see \cite[Section 2.2]{FOT}).
$\SSr(D)$ denotes the set of all  measures $\mu\in\SSr(E)$ such that $\mu(E\setminus D)=0$.
Let $\mu$ be a signed Borel measure on $E$, and let
$|\mu|=\mu^{+}+\mu^-$, where $\mu^+$ (resp. $\mu^-$) denotes the
positive (resp. negative) part of %the Jordan decomposition
$\mu$. We say that $\mu$ is smooth if $|\mu|\in\SSr(E)$. We denote by
$\MMr_b(D)$ the set of all signed Borel measures on $D$ such that
$\|\mu\|_{TV}:=|\mu|(D)<\infty$, and by $\MMr_{0,b}(D)$ the subset of
$\MMr_b(D)$ consisting of all smooth measures. Recall that by \cite[Lemma 2.1]{FST}, for every
$\mu\in\MMr_b(D)$ there exists a unique pair
$(\mu_d,\mu_c)\in\MMr_{0,b}(D)\times\MMr_b(D)$ such that
$\mu_c$ is concentrated on some $\EE^D$-exceptional Borel subset of $D$
and
\[
\mu=\mu_c+\mu_d.
\]
The measure $\mu_c$ (resp. $\mu_d$) is called the {\em concentrated} (resp. {\em diffusion}) part of $\mu$. For a complete description of the structure of $\mu_c$ see \cite{BGO} for the case of Laplace operator and \cite{KR:BPAN} for the general case.

Let $\nu$ be a Borel measures on $E$ and $f\in \mathcal B(E)$.
To shorten notation,  in what follows we denote
$\langle f,\nu\rangle=\langle \nu,f\rangle:=\int_E f\,d\nu$ whenever the integral exists.

\section{Orthogonal projections and Poisson kernels}

In what follows, we denote by  $qC(E)$ (resp. $qC(D)$) the family of all quasi continuous functions on $E$ (resp. $D$).

Recall that a set $V\subset E$   is called {\em quasi open} if for any $\varepsilon>0$
there exists an open set $G_\varepsilon$ containing $V$ with $\mbox{Cap}(G_\varepsilon\setminus V)<\varepsilon$, where  $\mbox{Cap}$ is the capacity associated with $\EE$.
We denote by $\mathcal O_q$ the family of all quasi open nearly Borel subsets of $E$,
and by $\mathcal O$ the family of all open subsets of $E$.  Clearly $\mathcal O\subset \mathcal O_q$.
Note that $u:E\to\bar\BR$ belongs to $qC(E)$ if and only if $u$ is finite q.e.
and $u^{-1}(I)$ is a quasi open set for any open set $I\subset \mathbb R$
(see the comments preceding  \cite[Lemma 2.1.5]{FOT}).  We denote by  $\BBr^*(E)$
the $\sigma$-algebra of {\em universally measurable} subsets of $E$.
A set $B\subset E$ belongs to $\BBr^*(E)$ if for any probability measure $\mu$ on $\BBr(E)$
there exist $B_1,B_2\in \BBr(E)$ such that $B_1\subset B\subset B_2$ and $\mu(B_2\setminus B_1)=0$.
Note that $\BBr^n(E)\subset \BBr^*(E)$.

For a quasi open $V\subset E$ we set
\[
F(V)=\{u\in F: u=0\mbox{ q.e. on }V^c:=E\setminus V\}.
\]
$F(V)$ is a closed linear subspace of $F$. We denote by $F(V)^{\bot}$ the orthogonal
complement of $F(V)$ in $F$ and by $\pi_V$ the orthogonal projection onto the space $F(V)$:
\[
F=F(V)\oplus F(V)^{\bot},\qquad \pi_V:F\rightarrow F(V).
\]
For $g \in F$ we set
\[
h_V(g)=g-\pi_V(g).
\]
Then $h_V(g)\in F(V)^{\bot}$ and, since $h_V(g)-g\in F(V)$,
\[
h_V(g)=g\quad\mbox{q.e. on }V^c.
\]

For  $U\in \mathcal O_q$ we set $\mathcal O_q(U):=\{V\in \mathcal O_q: V\subset U\}$.

\begin{definition}
We say that a family $\{P(x,dy),\, x\in E\}$ is a {\em sub-stochastic kernel on } $E$  if
\begin{enumerate}[{\rm (a)}]
\item $x\mapsto P(x,B)$ is universally measurable for any $B\in\BBr(E)$,
\item for each $x\in E$, $\BBr(E)\ni B\mapsto P(x,B)$ is a  positive   measure with $P(x,E)\le 1$.
\end{enumerate}
\end{definition}

\begin{definition}
Let $W\in\mathcal O_q$.  We say that a sub-stochastic   kernel  $\{P(x,dy),\, x\in E\}$ on $E$ is  {\em smooth (diffuse) on } $W$  if
 for each $x\in W$, $\BBr(E)\ni B\mapsto P(x,B)$ is a  diffuse  measure.
\end{definition}

By \cite[Theorem 4.3.2]{FOT} there exists an  exceptional set $N\subset E$
and a family of sub-stochastic kernels
$\{P_V(x,dy)$, $x\in E,\, V\in\OO_q \}$, that are  diffuse   on    $V\setminus N$ and supported in $V^c$
for any $x\in V\setminus N$, such that for every $g\in F$,
\[
h_V(g)(x)=\int_{V^c}g(y)\,P_V(x,dy)\quad \mbox{q.e. }x\in E.
\]
For each  $g\in\BB^+(E)$ (or  $g\in\BB_b(E)$) we let
\[
P_Vg(x)= \int_{V^c} g(y)\,P_V(x,dy),\quad x\in E\setminus N.
\]
For  $g\in\BB^+(E)$ we now let
\begin{equation}
\label{eq3.1}
\Pi_V(g)(x) =g(x)-P_V(g)(x),\quad x\in E\setminus N.
\end{equation}
By \cite[Theorem 4.3.2]{FOT} again (see also \cite{Sil}), $P_V(g)$ has the following  probabilistic interpretation:
for every $g\in \BB^+(E)$,
\begin{equation}
\label{eq3.2}
P_V(g)(x)=\mathbb E_x g(X_{\tau_V}),\quad x\in E\setminus N.
\end{equation}
It follows in particular that for every $B\in\BBr(V^c)$,
\[
P_V(x,B)=P_x(X_{\tau_V}\in B),\quad x\in E\setminus N,
\]
so $P_V(x,dy)$ is the distribution of the random variable $X_{\tau_V}$
provided that the process starts form $x$.
Clearly $P_V(x,dy)$ is concentrated on $V^c$, but if $x\in V$ and $X$ has continuous sample paths,  i.e. when $\EE$ is local (see \cite[Theorem 4.5.1]{FOT}), it is concentrated
on the Euclidean boundary $\partial V$.  Note also that by (\ref{eq2.1}), for any $g\in \BB^+(E)$ and $V\in\OO_q$\,,
\[
P_V(g)(x)=g(x)\quad \mbox{q.e. }x\in E\setminus V.
\]

Before formulating the next result let us recall (see \cite[Section 5.1]{FOT} for details)
that there is a one-to-one correspondence (so called Revuz duality) between
positive continuous additive functionals (PCAF) of $\BM$ and positive smooth measures.
For any $\nu\in \SSr(E)$ we denote by  $A^\nu$ the PCAF of $\BM$ in Revuz duality with $\nu$.
Furthermore, for any Borel measure such that  $|\nu|\in\SSr(E)$  we let $A^\nu= A^{\nu^+}-A^{\nu^-}$.
For $\nu\in \SSr(E)$
and $W\in\mathcal O_q$, we let
\[
R^W\nu(x):= \mathbb E_xA^\nu_{\tau_W},\quad x\in E.
\]
This notion agrees with \eqref{eq.not1} in case $W\in \mathcal O$ (see \cite[Theorem 5.1.3]{FOT}).

For $W\in \mathcal O_q$ we set
\[
\RRr(W)=\{\mu:|\mu|\in\SSr(E), R^W|\mu|<\infty\mbox{ q.e.}\}.
\]
Elements of $\RRr(W)$ may be called smooth (signed) measures of finite potential on $W$.
By \cite[Proposition 3.2]{KR:CM} applied to the form $\EE^W$ we have $\MMr_{0,b}(W)\subset\RRr(W)$.
% \textcolor{red}{(Potrzebny komentarz dotyczacy okreslenia $R^W\mu$; probabilistycznie?}

The following two simple lemmas will be useful.

\begin{lemma}
\label{lem2.4}
Let $V,W\in\mathcal O_q $ and $V\subset W$. If $\mu\in\RRr(W)$, then
$\Pi_V(R^W\mu)=R^V\mu$ q.e.
\end{lemma}
\begin{proof}
Without loss of generality (see the definition of the space $\mathscr S(E)$), we may assume that $\mu\ge0$ and $\mu\in F^*$.
Let $\eta\in F(V)$. Then
\[
\EE(R^W\mu,\eta)=\int_V\eta\,d\mu=\EE(R^V\mu,\eta).
\]
Hence $\EE(R^W\mu-R^V\mu,\eta)=0$ for $\eta\in F(V)$, which implies that $\Pi_V(R^W\mu-R^V\mu)=0$ q.e. As a result,   $\Pi_V(R^W\mu)=R^V\mu$ q.e.
\end{proof}
Note that Lemma \ref{lem2.4} is a slight generalization of Dynkin's formula (see \cite[(4.4.3)]{FOT}).

\begin{lemma}
\label{lem2.5}
Let $g\in F$. If $V,W\in\mathcal O_q $ and $V\subset W$, then
$P_V(P_W(g))=P_W(g)$ q.e.
\end{lemma}
\begin{proof}
Set $w=\Pi_V(P_W(g))$. Since $\Pi_V$ is a self-adjoint (as a projection)
operator and $w\in F(V)\subset F(W)$, $P_W(g)\in F(W)^{\bot}$, we have
\[
\EE(w,w)=\EE(P_W(g),\Pi_V(P_W(g)))=0,
\]
which implies the  desired result.
\end{proof}

\begin{corollary}
\label{cor2.6}
For any  $V,W\in\mathcal O_q $ such that $V\subset W$ we have
\[
P_V(x,dz)P_{W}(z,dy)=P_W(x,dy)\quad\text{for q.e. }x\in E.
\]
\end{corollary}
\begin{proof}
Set $\mu_x(dy)=P_V(x,dz)P_{W}(z,dy)$ and $\nu_x(dy)=P_W(x,dy)$.
By Lemma \ref{lem2.5}, for any $f\in C_c(E)\cap F$, $\langle\mu_x,f\rangle=\langle\nu_x,f\rangle$ for q.e. $x\in E$
(we use  separability of $C_c(E)$).
Since $(\EE,\mathfrak D(\EE))$ is regular, using an approximation argument we get
the above equality for all $f\in C_c(E)$.  This implies the desired result.
\end{proof}

\section{The space $\DD^1$ and its properties}
\label{sec4}

\subsection{Definition and basic properties.}
Let  $U\in\mathcal O_q$ and
\[
\Theta_U=\{\rho:U\rightarrow\BR,\mbox{ $\rho$ is strictly positive,
$\|\rho\|_{L^1(U;m)}=1$}\}.
\]
For $\rho\in\Theta_U$ we define the space $\DD^1_{\rho}(U)$ by
\[
\DD^1_{\rho}(U)=\{u\in\BB^n(E):  u=0 \text{ q.e. on } E\setminus U \text{ and } \lim_{n\rightarrow\infty}
\|(|u|-n)^+\|_{\DD^1_\rho(U)}=0\}.
\]
where
\[
\|u\|_{\DD^1_\rho(U)}=\sup_{V\in\OO_q(U)}\|P_V(|u|)\|_{L^1_{\rho}(U; m)}.
\]
We also let $L^1_\rho(U;m)$ denote the space of measurable functions $f$ on $U$
such that $\int_U|f|\rho\,dm<\infty$.

\begin{remark}
If $u\in\DD^1_{\rho}(U)$, then $\|u\|_{\DD^1_{\rho}(U)}=\sup_{V\in\OO_q(U)}\mathbb E_{\rho\cdot m}|u|(X_{\tau_V})<\infty$. The equality is immediate from (\ref{eq3.2}). Furthermore, if
$u\in\DD^1_{\rho}(U)$, then
\[
c_N:=\sup_{V\in\OO_q(U)}\mathbb E_{\rho\cdot m}(|u|-N)^+=\sup_{V\in\OO_q(U)}\mathbb E_{\rho\cdot m}[(|u|-N)\fch_{\{|u|>N\}}(X_{\tau_V})]<\infty
\]
for some $N\ge1$. Since $\rho \in\Theta_U$,  it follows that
\[
 \sup_{V\in\OO_q(U)}\mathbb E_{\rho\cdot m}[|u|(X_{\tau_V})]\le c_N+N.
 \]
\end{remark}

We will also need the following spaces:
\[
\DD_\rho^{1,c}(U)= \DD^1_{\rho}(U)\cap qC(U),\qquad \DD^1(U)=\bigcup_{\rho\in \Theta_U}\DD^1_\rho(U), \qquad
\DD^{1,c}(U)=  \DD^1(U)\cap qC(U).
\]
In case $U=E$ we omit $E$ in the notation. In the sequel, for $U\in\OO$ we will denote by $\TT_U$  the set of all $(\FF_t)$-stopping times $\tau$ such that $\tau\le\tau_U$.

\begin{proposition}
\label{prop4.1}
Let $U\in \mathcal O_q$ and $\rho\in\Theta_U$. If $u\in \DD^{1,c}_\rho(U)$, then
\[
\|u\|_{\DD^1_{\rho}(U)}=\|\sup_{V\in\OO_q(U)}P_V(|u|)\|_{L^1_{\rho}(U; m)}.
\]
Furthermore,
\[
\|u\|_{L^1_\rho(U)}\le  \|u\|_{\DD^1_{\rho}(U)}.
\]
\end{proposition}
\begin{proof}
We may and will assume that $u$ is nonnegative. Clearly, we have
\begin{equation}
\label{eq4.1}
\|u\|_{\DD^1_\rho(U)}\le \|\sup_{V\in\OO_q(U)}P_V(u)\|_{L^1_\rho(U; m)}.
\end{equation}
To show the opposite inequality we will frequently use relation   \eqref{eq3.2}  without special mention.
To simplify the notation in the remainder of the proof we shall  omit the subscript $U$ in $\TT_U$.
Set $w(x)=\sup_{\tau\in\TT}\mathbb E_xu(X_\tau)$, $x\in E$.
First suppose  that $u$ is bounded. By \cite[Theorem 2.41, page 140]{EK}, for any $\nu\in\PP$,
\[
\sup_{\tau\in\mathcal T}\mathbb E_\nu u(X_\tau)=\mathbb E_\nu u(X_{\tau^*}),
\]
where $\tau^*=\inf\{t\ge 0: w(X_t)=u(X_t)\}\wedge\tau_U$. Observe that $\tau^*=\tau_{V^*}$, where $V^*=\{w>u\}\cap U$.
As a result
\begin{equation}
\label{eq4.2}
\sup_{\tau\in\mathcal T}\BE_\nu u(X_\tau)=\sup_{V\in\OO_q(U)} \mathbb E_\nu u(X_{\tau_V})=\BE_\nu u(X_{\tau_{V^*}}).
\end{equation}
Consequently,
\begin{align}
\label{eq4.3}
\|\sup_{V\in \mathcal O_q(U)}P_V(u)\|_{L^1_{\rho}(U; m)}
&=\int_U\sup_{V\in\OO_q(U)}\BE_x u(X_{\tau_V})\,\rho(x)\,dx\nonumber \\
&=\int_U \mathbb E_x u(X_{\tau_{V^*}})\,\rho(x)\,dx=\mathbb E_{\rho\cdot m}u(X_{\tau_{V^*}})\nonumber \\
&=\sup_{V\in\OO_q(U)}\mathbb E_{\rho\cdot m}u(X_{\tau_V})
=\sup_{V\in\OO_q(U)}\|P_V(u)\|_{L^1_{\rho}(U;m)}.
\end{align}
To show the general case, write $u_n= u\wedge n$.
By \eqref{eq4.3},
\[
\int_U\sup_{V\in \mathcal O_q(U)} P_V(u_n)(x)\,\rho(x)\,dx=\sup_{V\in\OO_q(U)}\int_U  P_V(u_n)(x)\,\rho(x)\,dx.
\]
Applying  Fatou's lemma we get
\begin{align*}
\int_U\sup_{V\in \mathcal O_q(U)}P_V(u)(x)\,\rho(x)\,dx&\le \liminf_{n\to \infty} \int_U\sup_{V\in \mathcal O_q(U)}P_V(u_n)(x)\,\rho(x)\,dx\\
&\le \sup_{V\in \mathcal O_q(U)}\int_U  P_V(u)(x)\,\rho(x)\,dx,
\end{align*}
which together with \eqref{eq4.1} gives the asserted equality.

As to the inequality claimed in the proposition, recall that by \cite[Theorem 4.2.2]{FOT},
$[0,\tau_U)\ni t\mapsto u(X_t)$ is right continuous under the measure $P_x$ for a.e. $x\in U$.
Hence, by Fatou's lemma, $\lim_{\varepsilon\to 0}\mathbb E_x u(X_{\tau_{B(x,\varepsilon)}})\ge u(x)$
a.e. As a result, $\sup_{V\in \mathcal O_q(U)}P_V(u)(x)\ge u(x)$ a.e. This finishes the proof.
\end{proof}

\begin{lemma}
\label{lem4.2}
The following assertions hold true for any $U\in\mathcal O$ and any $\rho\in \Theta_U$:
\begin{enumerate}[\rm(i)]
\item $\mathcal B^n_b(U)\subset \DD^1_\rho(U)$ and   $\overline{\BB_b(U)}^{\DD^1_\rho(U)}=\DD^1_\rho(U)$.

\item  Suppose that $u\in\BB(U)$ and there exists a nonnegative measure $\nu\in\RRr(U)$
such that $|u|\le R^U\nu$ q.e. Then $u\in\DD^1(U)$ and $\|u\|_{\DD^1_\rho}\le \int_DR^U\rho\,d\nu$.

\item $F(U)\subset \DD^{1,c}(U)$ and   $\|u\|_{\DD^1_\rho(U)}\le \|u\|_F\sqrt{(\rho, R^U\rho)}$ for $u\in F(U)$.
%\item If $g\in\mathcal B^n(E)$ and $P_U(|g|)<\infty$ q.e., then $P_U(g)\in \DD^1(U)$.
\end{enumerate}
\end{lemma}
\begin{proof}
(i) The first assertion  is obvious. As for the second one, if $u\in\DD^1(U)$, then $T_nu:=((-n)\vee u)\wedge n\in\BB^n_b(U)$ and $u_n\rightarrow u$ in $\DD^1_{\rho}$ as $n\rightarrow\infty$ since $|u-T_nu|=(|u|-n)^+$, $n\ge1$.\\
%K\\
%Z definicji normy w $\DD^1$ i powyższej rownosci mamy %$\|u-u_n\|_{\DD^1_{\rho}}=\||u-T_nu|\|_{D^1_{\rho}}=\|(|u|-n)^+\|_{D^1_{\rho}}$ co zbiega %do 0 bo $u\in\DD^1$.
%KK\\
(ii) By the 0-order version of \cite[Theorem 2.2.4]{FOT} (see the comments following \cite[Corollary 2.2.2]{FOT}), there exists
an increasing sequence $\{F_n\}$ of closed subsets of $U$ such that
$\mbox{Cap}_{\EE_U}(K\setminus F_n)\to 0$ for any compact $K\subset U$,
$\mathbf1_{F_n}\cdot\nu\in F^*$ and $\|R^U(\mathbf1_{F_n}\cdot\nu)\|_\infty<\infty$, $n\ge1$.
Let $\rho\in \Theta_U$ be such that $\int_U(R^U\nu)\rho\,dm<\infty$.
By Lemma \ref{lem2.4}, $\Pi_V(R^U\mu)\ge 0$ for any quasi open set $V\subset U$ and
$\mu\in\SSr(U)$. By this and (\ref{eq3.1}),
\[
\int_UP_V(R^U\nu- R^U(\mathbf1_{F_n}\cdot\nu))\rho\,dm \le \int_U R^U(\mathbf1_{U\setminus F_n}\cdot\nu)\rho\,dm.
\]
%If (2) or (3) is satisfied, then clearly the right-hand side of \eqref{eq2.6} tends to zero when  $n\to \infty$. On the other hand, under (1), we have
%\[
%\int_E R(\mathbf1_{F^c_n}\cdot\nu)\rho\,dm=\EE(R(\mathbf1_{F^c_n}\cdot\nu),R\rho)
%\le \Big(\int_E\mathbf1_{F_n^c}R(\mathbf1_{F^c_n}\cdot\nu)\,d\nu\Big)^{1/2}  \Big(\int_E \rho %\rho\,dm\Big)^{1/2}\to 0,
%\]
%as $n\to \infty$.
The right-hand side of the above inequality tends to zero when  $n\rightarrow\infty$.
Hence
\[
\lim_{n\rightarrow\infty}\|R^U\nu-R^U(\mathbf1_{F_n}\nu)\|_{\DD^1_\rho(U)}=0.
\]
By the choice of $\{F_n\}$, we have $(R^U(\mathbf1_{F_n}\nu))\subset \BB^n_b(U)$ for $n\ge1$.
 By this and (i), $R^U(\mathbf1_{F_n}\nu)\in\DD^1_{\rho}(U)$, $n\ge1$.  Consequently,
$R^U\nu\in \DD^1_\rho(U)$, so $u\in\DD^1(U)$.
\\
(iii) That $F(U)\subset\DD^{1,c}(U)$ follows from part (ii) and the fact that by \cite[Theorem 2.2.1]{FOT}, for any $u\in F(U)$ there exists a nonnegative $\nu\in\RRr(U)$ such that
$|u|\le R^U\nu$. The asserted inequality follows from \cite[Lemma 5.1.1]{FOT}.
\end{proof}

\subsection{Equivalent definitions.}
The 
following theorem supplies the key to the   characterizations of $\DD^{1,c}$ mentioned in the introduction.

\begin{theorem}
\label{th4.3}
Let $U\in\mathcal O$ and  $\rho\in \Theta_U$.
\begin{enumerate}[\rm(i)]
\item If   $u\in\DD^{1,c}_{\rho}(U)$, then the family $\{u(X_\tau),\, \tau\in\mathcal T_U\}$
is uniformly integrable under the measure $P_x$ for q.e. $x\in U$ and under the
measure $P_{\rho\cdot m}$.
\item If $u\in qC(U)$ and $\{u(X_{\tau_V}),\, V\in\mathcal O_q(U)\}$  is uniformly integrable under the measure $P_{\rho\cdot m}$,
then  $u\in\DD^{1,c}_{\rho}(U)$.
\end{enumerate}
\end{theorem}
\begin{proof}
By Proposition \ref{prop4.1},
\begin{equation}
\label{eq4.4}
\lim_{n\to \infty}\|\sup_{V\in\OO_q}P_V[(|u|-n)^+]
\|_{L^1_{\rho}(E; m)}=0.
\end{equation}
Let $w_n(x):=\sup_{V\in\OO_q}P_V[(|u|-n)^+](x)$. The above convergence implies that $w_n\searrow 0$ $m$-a.e. By \eqref{eq4.2},
\begin{equation}
\label{eq.dwkr}
w_n(x)=\sup_{\tau\in\mathcal T_U}\mathbb E_x\big((|u|-n)^+(X_\tau)\big),\quad x\in E.
\end{equation}
It follows that  $w_n$ is an $(P^U_t)$-excessive function.
Since finite $m$-a.e. excessive functions  are quasi continuous (see \cite[Theorem A.2.7, Theorem 4.6.1]{FOT}), we see that in fact  $w_n\searrow 0$ q.e.,
which shows that the family $\{u(X_\tau),\, \tau\in\mathcal T_U\}$
is uniformly integrable under the measure $P_x$ for q.e. $x\in U$.  That this family is uniformly integrable under  $P_{\rho\cdot m}$ is an easy
consequence of \eqref{eq4.4} and \eqref{eq.dwkr}. Conversely, if $\{u(X_{\tau_V}),\, V\in\mathcal O_q(U)\}$
is uniformly integrable under the measure $P_{\rho\cdot m}$, then
\[
\sup_{V\in\mathcal O_q} \mathbb E_{\rho\cdot m}P_V (|u|-n)^+\to 0,
\]
which means that $u\in\DD^{1,c}_\rho(D)$.
\end{proof}

Let $\rho\in\Theta_D$. Recall that by the de la Vall\'ee  theorem (see, e.g., \cite[Chapter II, Theorem 22]{DM}), a subset $\KK$ of $L^1_{\rho}(D; m)$ is uniformly integrable if and only if there exists $\varphi\in$\,FVP (see Introduction) such that $\sup_{u\in\KK}\|\varphi(|u|)\|_{L^1_\rho}<\infty$.

\begin{theorem}
\label{th4.5}
Let $u\in qC(D)$ and $\rho\in \Theta_D$. Then $u\in\DD^{1,c}_\rho(D)$ if and only if there exists $\varphi\in\mbox{\rm FVP}$  such that
\begin{equation}
\label{eq4.5}
\sup_{V\in \mathcal O_q(D)} \|P_V\varphi(|u|)\|_{L^1_{\rho}(D; m)}<\infty.
\end{equation}
\end{theorem}
\begin{proof}
If $u\in\DD^{1,c}_{\rho}(D)$, then by Theorem \ref{th4.3}(i) the family
$\KK:=\{u(X_{\tau_V}),\, V\in\mathcal O_q(D)\}$  is uniformly integrable under $P_{\rho\cdot m}$. Hence $\sup_{V\in\OO_q}\mathbb E_{\rho\cdot m}\varphi(|u(X_{\tau_V})|)<\infty$ for some $\varphi\in\mbox{FVP}$, which shows (\ref{eq4.5}) by (\ref{eq3.2}). Conversely, if (\ref{eq4.5}) is satisfied for some $\varphi\in\mbox{FVP}$, then by (\ref{eq3.2}) and the de la Vall\'ee--Poussin theorem, $\KK$  is uniformly integrable. Hence
$u\in\DD^{1,c}_{\rho}(D)$ by  Theorem \ref{th4.3}(ii).
\end{proof}

\begin{definition}
Let $(\GG_t)_{t\ge0}$ be a filtration. A $(\GG_t)_{t\ge0}$-adapted stochastic process
$Y$ is of class (D)  if the collection $\{Y_{\tau}: \tau$ a finite valued
$(\GG_t)_{t\ge0}$-stopping time\} is uniformly integrable.
\end{definition}

The name ``class (D)" was given by P. A. Meyer. According to \cite[p. 107]{P}, presumably he expected it to come to be known as ``Doob class" at some point, but it has stayed class (D).

\begin{corollary}
$u\in\DD^{1,c}(D)$  if and only if the process $ [0,\tau_D)\ni t\mapsto u(X_{t})$
is right continuous and  of class \mbox{\rm (D)} under the measure $P_x$ for q.e. $x\in D$.
\end{corollary}
\begin{proof}
Necessity follows readily from
%Let $u\in\DD^{1,c}_\rho(D)$. That  then $u(X)$ is of class (D) is immediate from
Theorem \ref{th4.3}(i) and \cite[Theorem 4.2.2]{FOT}.
As for the sufficiency part, quasi-continuity of $u$ follows from \cite[Theorem A.2.7, Theorem 4.6.1]{FOT}. Furthermore, by the assumptions, \eqref{eq.dwkr} tends to zero q.e., which implies \eqref{eq4.4}
for some $\rho\in \Theta_D$.
\end{proof}

\subsection{Further properties.}

Let $\{u_n,u\}\subset\BB(D)$. We say that $\{u_n\}$ converges to $u$ {\em $\EE^D$-quasi uniformly} (resp. {\em $\EE^D$-quasi uniformly on compacts})
if for every $\varepsilon>0$ there exists a closed set $F_\varepsilon\subset D$ such that $\mbox{\rm Cap}_D(D\setminus F_\varepsilon)\le\varepsilon$
and $\sup_{x\in F_\varepsilon}|u_n(x)-u(x)|\to 0$. (resp. $\sup_{x\in F_\varepsilon\cap K}|u_n(x)-u(x)|\to 0$ for any compact $K\subset D$).

\begin{lemma}
\label{lem4.10}
Let $\{u_n\}\subset qC(D)$ be  such that $u_n\searrow 0$ q.e.
Then  $u_n\searrow 0$ $\EE^D$-quasi uniformly on compacts.
\end{lemma}
\begin{proof}
By \cite[Theorem 2.1.2]{FOT}, there exists an increasing family $\{F_k\}$ of closed subsets of $E$
such that $\mbox{Cap}(E\setminus F_k)\le 1/k$ and ${u_n}|_{F_k}$ is continuous for any $n,k\ge 1$. By Dini's theorem ${u_n}|_{F_k}\searrow 0$ uniformly on compacts.
From this one easily deduces the assertion.
\end{proof}

In \cite{LeJan} the above notion is considered but with respect to the  capacity defined, for some $\alpha>0$, by
$\mbox{Cap}_\alpha(A):=\mathbb E^D_{\rho\cdot m}e^{-\alpha\tau_A}$ for  $A\in\mathcal B^n(E)$.
By \cite[Theorem 4.2.5,Theorem 2.1.5]{FOT} and Lemma \ref{lem4.10}, if $U_n$ is a nonincreasing
sequence of quasi open sets such that $\mbox{Cap}_D(U_n)\searrow 0$, then $\mbox{Cap}_\alpha(U_n)\searrow 0$ for any $\alpha>0$. Conversely,
if $\mbox{Cap}_\alpha(U_n)\searrow 0$ for some (hence for any) $\alpha>0$, and $\mbox{Cap}_D(U_{n_0})<\infty$ for some $n_0\ge 0$, then  $\mbox{Cap}_D(U_n)\searrow 0$.

\begin{lemma}
\label{lem4.8}
Let $\{u_n\}\subset qC(D)$ be a sequence
such that for every $T\ge 0$,
\[
\sup_{0\le t\le \tau_D\wedge T}|u_n(X_t)|\to 0\quad P_{\rho\cdot m}\text{-a.s.}
\]
Then $u_n\to 0$ $\EE^D$-quasi uniformly on compacts.
\end{lemma}
\begin{proof}
Follows from the comments preceding \cite[Theorem 1]{LeJan} and the comments preceding
the lemma.
\end{proof}

Note that, by  \cite[Remark 2.1]{KRS} applied to $u(X)$, for any $\alpha\in (0,1)$ and  $u\in \DD^1_{\rho,c}(D)$,
\begin{equation}
\label{eq4.6}
\mathbb E_{\rho\cdot m}\sup_{t\le \tau_D}|u(X_t)|^{\alpha}\le \frac{1}{1-\alpha}\|u\|^{\alpha}_{\DD^1_\rho(D)}.
\end{equation}

\begin{proposition}
\label{prop4.11}
The set  $\DD^{1,c}_{\rho}(D)$ with the norm
$\|\cdot\|_{\DD^1_\rho(D)}$ form a Banach space.
\end{proposition}
\begin{proof}
Suppose that $\{u_n\}\subset \DD^{1,c}_{\rho}(D)$ is a Cauchy sequence,  i.e. for any $\varepsilon>0$ there is $N_\varepsilon\ge1$ such that
\begin{equation}
\label{eq4.7}
\|u_n-u_m\|_{\DD^1_\rho(D)}\le \varepsilon,\quad n,m\ge N_\varepsilon.
\end{equation}
By \eqref{eq4.6},
\[
\BE_{\rho\cdot m}\sup_{t\le \tau_D}|u_n(X_t)-u_m(X_t)|^{1/2}\le 2\|u_n-u_m\|^{1/2}_{\DD^1_\rho(D)},\quad n,m\ge 1.
\]
By Lemma \ref{lem4.8},  $\{u_n\}$ is convergent (up to a subsequence) $\EE^D$-quasi uniformly on compacts.
Hence $u:=\lim_{n\to\infty} u_n$ q.e. is quasi continuous on $D$.
By \cite[Theorem 4.1.1, Theorem 4.2.1]{FOT}, for any $V\in\OO_q $, $u_n(X_{\tau_V})\to u(X_{\tau_V})$ $P_x$-a.s. for q.e. $x\in E$.
Therefore applying Fatou's lemma  we conclude from \eqref{eq4.7} that
$\|u_n-u\|_{\DD^1_\rho(D)}\le \varepsilon$ for $n\ge N_\varepsilon$,
which implies  the required result.
\end{proof}

We denote by  $P^\alpha_V$  the operator constructed in the same way as $P_V$
but for the Dirichlet form $\EE_\alpha:=\EE+\alpha(\cdot,\cdot)$. By \cite[Theorem 2.1.6]{FOT}, the families of quasi open sets corresponding to both
forms, i.e. $\EE$ and $\EE_\alpha$, coincide.
By \cite[Theorem 4.3.1]{FOT}, for any $u\in\mathcal B^+(E)$ and  $V\in\OO_q $ we have
\[
P^\alpha_V(u)=\mathbb E_x[e^{-\alpha\tau_V}u(X_{\tau_V})].
\]
For any $u\in \mathcal D^{1,c}_\rho(D)$ we let
\[
e^\alpha_{u}(x):=
\sup_{V\in\mathcal O_q(D)}\BE_{x} e^{-\alpha\tau_V}|u(X_{\tau_V})|,\quad x\in D.
\]

The following proposition shows that under an additional regularity assumption on the resolvent, in the definition of $\DD^1_{\rho}(D)$ one can replace the family $\OO_q$ by $\OO$.

\begin{proposition}
\label{prop4.8}
Assume that $R_1(C_b(D))\subset C_b(D)$. Then for every
$u\in{\DD^{1,c}_\rho}(D)$,
\[
\|u\|_{\DD^1_\rho(D)}=\|\sup_{V\in \mathcal O(D)}P_V(|u|)\|_{L^1_{\rho}(D; m)}
=\sup_{V\in\OO(D)}\|P_V(|u|)\|_{L^1_{\rho}(D; m)}.
\]
\end{proposition}
\begin{proof}
Let $u\in\DD^{1,c}_\rho(D)$. Without loss of generality (by the very definition of the space $\DD^{1,c}_\rho(D)$)
we may and will assume that $u\in\mathcal B_b^+(D)$.
By \cite{LeJan}, for every $\alpha>0$ there exists a sequence $\{u_n\}\subset C_c^+(D)$
such that
\[
\BE_{\rho\cdot m}[\sup_{t\le \tau_D}e^{-\alpha t}|u-u_n|(X_t)]\to 0.
\]
This implies that
\[
\sup_{V\in\mathcal O_q(D)}\BE_{\rho\cdot m} e^{-\alpha\tau_V}|u_n(X_{\tau_V})|\to \sup_{V\in\mathcal O_q(D)}\BE_{\rho\cdot m} e^{-\alpha\tau_V}|u(X_{\tau_V})|,
\]
\[
\sup_{V\in\mathcal O(D)}\mathbb E_{\rho\cdot m} e^{-\alpha\tau_V}|u_n(X_{\tau_V})|\to \sup_{V\in\mathcal O(D)}\mathbb E_{\rho\cdot m} e^{-\alpha\tau_V}|u(X_{\tau_V})|
\]
and
\[
\int_D\sup_{V\in\mathcal O(D)}\mathbb E_xe^{-\alpha\tau_V}|u_n(X_{\tau_V})|\,\rho(x)\,m(dx)\to
\int_D\sup_{V\in\mathcal O(D)}\mathbb E_x e^{-\alpha\tau_V}|u(X_{\tau_V})|\rho(x)\,m(dx).
\]
Assume for a moment that  $e^\alpha_{u_n}\in C_b(D)$. Then
\begin{align*}
\sup_{V\in\mathcal O_q(D)}\BE_{\rho\cdot m} e^{-\alpha\tau_V}|u_n(X_{\tau_V})|&=
\BE_{\rho\cdot m} e^{-\alpha\tau_{V^*}}|u_n(X_{\tau_{V^*}})|=
\sup_{V\in\mathcal O(D)}\BE_{\rho\cdot m} e^{-\alpha\tau_V}|u_n(X_{\tau_V})|\\&
=\int_D\sup_{V\in\mathcal O(D)}\BE_xe^{-\alpha\tau_V}|u_n(X_{\tau_V})|\,\rho(x)\,m(dx).
\end{align*}
This when combined with the already proven  convergences  gives
\begin{align*}
\sup_{V\in\mathcal O_q(D)}\BE_{\rho\cdot m} e^{-\alpha\tau_V}|u(X_{\tau_V})|
&=\sup_{V\in\mathcal O(D)}\mathbb E_{\rho\cdot m} e^{-\alpha\tau_V}|u(X_{\tau_V})|\\
&=\int_D\sup_{V\in\mathcal O(D)}\mathbb E_xe^{-\alpha\tau_V}|u(X_{\tau_V})|\,\rho(x)\,m(dx)
\end{align*}
for  $\alpha>0$. From this one easily gets the desired  result.
What is left is to show that indeed $e^\alpha_{u_n}\in C_b(D)$, $n\ge 1$.
By \cite[Corollary 3.15, Remark 4.2]{K:SPA} and \cite[Theorem 4.4]{KR:ALEA},
$w^n_k(x)\nearrow e^\alpha_{u_n}(x)$ and
\begin{equation}
\label{eq4.8}
|w^n_k(x)-e^\alpha_{u_n}(x)|\le \sup_{\tau\le\tau_D}\mathbb E_xe^{-\alpha\tau}(w^n_k-u_n)^-(X_\tau),
\end{equation}
where $w_k^n\in L^1(D;m)$ is the unique solution (see Definition \ref{def5.2}) to
\[
-Lw_k^n+\alpha w^n_k=k(w^n_k-u_n)^-\quad\text{in } D,\quad w_k^n=0\quad\text{on } D^c.
\]
Since $R_1^D(C_b(D))\subset C_b(D)$ one easily deduce that $w^n_k\in C_b(D)$.
Therefore, by Dini's theorem $(w^n_k-u_n)^-\searrow 0$, as $k\to\infty$, uniformly on compact subsets of $D$.
Let  $K\subset D$ be a compact set that supports $u_n$. Then
\[
\begin{split}
\sup_{\tau\le\tau_D}\mathbb E_xe^{-\alpha\tau}(w^n_k-u_n)^-(X_\tau)&=
\sup_{\tau\le\tau_D}\mathbb E_xe^{-\alpha\tau}[\mathbf1_K(X_\tau)(w^n_k-u_n)^-(X_\tau)]\\&\le
\mathbb E_x\sup_{t\le \tau_D}e^{-\alpha t}[\mathbf1_K(X_t)(w^n_k-u_n)^-(X_t)]\le \sup_{x\in K}(w^n_k-u_n)^-(x).
\end{split}
\]
This together with \eqref{eq4.8} implies that $w^k_n\to u_n$, as $k\to \infty$, uniformly on compact subsets of $D$.
Since $w^k_n$ are continuous, we infer that $u_n$ is continuous as well.
\end{proof}

\section{Reconstruction formula by means of harmonic measures}
\label{sec5}

We recall that we assume the absolute continuity condition (\ref{eq2.6}). For a nonnegative Borel measure $\mu$ on $D$ we define $R^D\mu$ by (\ref{eq.not1}), and
and for a signed Borel measure $\nu$ on $D$ we set $R^D\nu(x)=R^D\nu^+(x)-R^D\nu^-(x)$
for $x\in D$ such that $R^D\nu^+(x)<\infty$ or $R^D\nu^-(x)<\infty$, and $R^D\nu(x)=0$ otherwise.

\begin{lemma}
\label{lem5.1}
Let $\mu\in\MMr_b(D)$. Then   $R^D|\mu|<\infty$ q.e. on $D$ and  $u=R^D\mu$ is a quasi-continuous function on $D$.
%$u=0$ q.e. on $D^c$?
\end{lemma}
\begin{proof}
For the proof that  $R^D|\mu|<\infty$ q.e. on $D$ see \cite[Proposition 3.2]{K:CVPDE}.
Since $\alpha R^D_{\alpha}R^D\mu^+(x)=\int_D\alpha R^D_{\alpha}r^D(x,y)\,\mu^{+}(dy)$ and
$r^D(\cdot,y)$ is excessive for each $y\in D$, applying \cite[Proposition II.(2.3)]{BG} and monotone convergence shows that $R^D\mu^{+}$ is excessive relative to $(P^D_t)_{t>0}$. Likewise, $R^D\mu^{-}$ is excessive. By this and \cite[Theorem A.2.7]{FOT} (or \cite[Proposition II.(4.2)]{BG}), $R^D\mu$ is finely continuous q.e.
Since we know that $R^D\mu$ is q.e. finite, it is quasi-continuous on $D$ by  \cite[Theorem 4.6.1]{FOT}.
\end{proof}

\begin{definition}
\label{def5.2}
Let $\mu\in\MMr_{b}(D)$. The function $u$ (defined q.e. on $E$) by
\[
u=R^D\mu\quad \mbox{q.e. on }D,\qquad u=0\quad\mbox{on }D^c,
\]
is called {\em integral solution} of (\ref{eq1.5}).
\end{definition}

From Lemma \ref{lem5.1} we know that for $\mu\in\MMr_{b}(D)$ the integral solution is well defined and is a quasi continuous function on $D$. We will also need the notion of probabilistic solution. Its definition requires some preparatory results.

In the remainder of this section we assume that
\begin{equation}
\label{eq5.4}
P_x(\tau_D<\infty)=1\quad\text{q.e. }x\in D.
\end{equation}
This condition holds e.g. provided that $D$ is relatively compact.

We say that a nondecreasing sequence $\{\tau_k\}$ of {\em stopping times} is a {\em reducing sequence} for a measurable function $u$
on $D$ if $\tau_k\nearrow  \tau_D\, P_x$-a.s. for q.e. $x\in D$ and
\[
\mathbb E^D_x\sup_{t\le\tau_k}|u(X_t)|<\infty,\quad k\ge 0,\quad\mbox{q.e. }x\in D.
\]

\begin{lemma}
\label{lm5.4}
Let $\mu\in\MMr_{b}(D)$ and $u=R^D\mu$, $w=R^D|\mu|$. Then $\{\tau_k\}$, where
\[
\tau_k=\inf\{t>0: w(X_t)>k\}\wedge\tau_D,\quad k\ge1,
\]
is a reducing sequence for $u$. Moreover, $P_x(\tau_k<\tau_D)\to 0$ for  q.e. $x\in D$.
\end{lemma}
\begin{proof}
Observe that
\[
\mathbb E_x \sup_{t\le \tau_k}|u(X_t)|\le \mathbb E_x \sup_{t\le \tau_k}w(X_t) \le k+\mathbb E_xw(X_{\tau_k}).
\]
On the other hand, by Fatou's lemma,
\[
\mathbb E_xw(X_{\tau_k})\le \liminf_{t\to\infty}\mathbb E_x w(X_{\tau_k\wedge t})\le \liminf_{t\to\infty}\mathbb E_x w(X_t)\le w(x)\quad\text{q.e.}
\]
We used here the fact that $w(X)$ is a supermartingale (see \cite[Theorem III.5.7]{BG}). Hence
\[
\mathbb E_x \sup_{t\le \tau_k}|u(X_t)|\le k+w(x)\quad\text{q.e.},
\]
which proves the first assertion of the lemma. By \cite[Lemma 2.4]{KR:NoD},
\[
\mbox{Cap}_{\EE^D}(w> k)\le k^{-1}\|\mu\|_{TV},\quad k\ge1,
\]
where $\mbox{Cap}_{\EE^D}$  denotes the $0$-order capacity introduced in \cite[page 74]{FOT}.
Hence, by \cite[Lemma 2.1.8, Theorem 4.2.1]{FOT}, $P_x(\tau_k<\tau_D)\to 0$ q.e. as $k\to\infty$. 
\end{proof}

\begin{definition}
A function $u\in qC(D)$ is called a {\em probabilistic solution} of \eqref{eq1.5} if
for q.e. $x\in D$ there exists a local martingale $M^x$ such that
\[
u(X_t)=u(x)-A^{\mu_d}_t+M^x_t,\quad x\in [0,\tau_D],\quad P^D_x\text{-a.s.}
\]
and  for any reducing sequence $(\tau_k)$ for $u$, $\mathbb E^D_xu(X_{\tau_k})\to R^D\mu_c(x)$ q.e.
\end{definition}

The notions of integral and probabilistic solutions are  equivalent. Namely, the following results was proved in \cite[Proposition 3.12]{K:CVPDE}:

\begin{proposition}
Let $\mu\in\MMr_b(D)$. Then $u$ is a probabilistic solution of \mbox{\rm(\ref{eq1.5})} if and only if it is its integral solution.
\end{proposition}

\begin{lemma}
\label{lm5.6}
Suppose that $u$ is a nonnegative probabilistic  solution of \eqref{eq1.5}. Let $(\tau_k)$
be a reducing sequence for $u$. Then for any $n\ge1$ and  $\nu\in \mathcal P$ such that $R^D\nu$ is bounded,  we have
\[
\mathbb E_\nu [(u-n)^+(X_{\tau_k})]\to (R^D\nu,\mu_c)\quad\text{as}\quad k\to\infty.
\]
\end{lemma}
\begin{proof}
Let $(\tau_k)$ be a localizing sequence for $u$ such that $P_x(\tau_k<\tau_D)\to 0$ q.e. as $k\to \infty$ (see Lemma \ref{lm5.4}).
Set $u_n=(u-n)^+$. By \cite[Proposition 6.2]{K:CVPDE} there exists  $\mu_n\in \MMr_b(D)$ such that
\begin{equation}
\label{eq5.1}
-Lu_n=\mu_n\quad\text{in }D,\quad u=0\text{ on } D^c.
\end{equation}
Clearly $u_n\le u$, so by \cite[Theorem 6.1]{K:CVPDE}, $(\mu_n)_c\le \mu_c$. On the other hand, by the very definition of a probabilistic solution of \eqref{eq1.5} and \eqref{eq5.1} we have
\[
\mathbb E_xu(X_{\tau_k})\to R^D\mu_c(x),\qquad \mathbb E_x(u-n)^+(X_{\tau_k})\to R^D[(\mu_n)_c](x)\quad \text{q.e. }x\in D.
\]
Consequently, for q.e. $x\in D$,
\[
\begin{split}
R^D[(\mu_n)_c](x) &=\lim_{k\to \infty } \mathbb E^D_x[(u-n)^+(X_{\tau_k})]
\ge \lim_{k\to \infty }  \mathbb E^D_x[(u-n)(X_{\tau_k})]\\&
=\lim_{k\to \infty }  \mathbb E^D_xu(X_{\tau_k})-n\lim_{k\to \infty }P_x(\tau_k<\tau_D)=R^D\mu_c(x).
\end{split}
\]
As a result, $R^D\mu_c=R^D[(\mu_n)_c] $ q.e., and hence $\mu_c=(\mu_n)_c$. Thus, in fact, for q.e. $x\in D$ we have
\[
u_n(x)=R^D\mu_c(x)+R^D(\mu_n)_d(x).
\]
Let $(\tau_k)$ be a reducing sequence for $u_n$.  By the definition of a probabilistic solution of \eqref{eq5.1},
\[
u_n(x)=\mathbb E_xu_n(X_{\tau_k})+\mathbb E_xA^{(\mu_n)_d}_{\tau_k}
\]
for q.e. $x\in D$. This implies that for any $\nu$ as in the formulation of the lemma we have
\[
\int_Du_n\,\nu(dx)=\mathbb E_\nu u_n(X_{\tau_k})+\mathbb E_\nu A^{(\mu_n)_d}_{\tau_k}.
\]
Note that $|\mathbb E_x A^{(\mu_n)_d}_{\tau_k}|\le \mathbb E_x A^{|(\mu_n)_d|}_{\tau_D}=R^D|(\mu_n)_d|(x)$ for q.e. $x\in D$.
Therefore applying the  dominated convergence theorem one easily shows the desired result.
\end{proof}

\begin{corollary}
\label{cor5.7}
Suppose that $u$ is a  probabilistic solution of \eqref{eq1.1}. Let $(\tau_k)$
be a reducing sequence for $u$. Then for any $n\ge1$ and any $\nu\in \mathcal P$ such that $R^D\nu$ is bounded  we have
\[
\mathbb E_\nu [(u-n)^+(X_{\tau_k})]\to (R^D\nu,\mu^+_c),\qquad \mathbb E_\nu [(u+n)^-(X_{\tau_k})]\to (R^D\nu,\mu^-_c) \quad \text{as}\quad k\to\infty.
\]
\end{corollary}

 The next theorem specifies  how the behaviour of the solution $u$ of (\ref{eq1.5}) on the set where $u$ i very large is related to the concentrated part
of $\mu$. One can call it the reconstruction formula for $\mu_c$. Another formula of this type will be given in Theorem \ref{th6.2}.
\begin{theorem}
\label{th5.8}
Let $\mu\in\MMr_b(D)$ and $u$ be the integral  solution of \mbox{\rm(\ref{eq1.5})}. Then
for any $\rho\in\Theta_D$ such that $R^D\rho$ is bounded we have
\[
\lim_{n\rightarrow\infty}\|(|u|-n)^+\|_{\DD^1_\rho(D)}=\int_D R^D\rho\,d|\mu_c|.
\]
As a result, $\mu$ is diffuse if and only if $u\in\DD^{1,c}(D)$.
\end{theorem}
\begin{proof}
Fix $\rho$ as in the formulation of the theorem. By Corollary \ref{cor5.7} and Lemma \ref{lm5.4},
\[
\mathbb E_{\rho\cdot m}(|u|-n)^+(X_{\tau_k})\to (R^D\rho,\mu_c),
\]
where $\tau_k=\tau_{V_k}$ with $V_k:=\{R^D|\mu|<k\}$. The set $V_k$ is quasi open since  $R^D|\mu|$ is quasi continuous.
Therefore, by the definition of the norm $\|\cdot \|_{\DD^1_\rho}$, we have
\begin{equation}
\label{eq5.2}
\|(|u|-n)^+ \|_{\DD^1_\rho(D)}\ge (R^D\rho,|\mu_c|).
\end{equation}
On the other hand, by Corollary \ref{cor5.7} and \cite[Proposition 6.2, (6.1)]{K:CVPDE},
\[
(|u|-n)^+\le R^D|\mu_c|+R^D(\mathbf1_{\{|u|>n\}}|\mu_d|)\quad\text{q.e.}
\]
By Lemma \ref{lem4.2}(ii),
\begin{equation}
\label{eq5.3}
\|(|u|-n)^+ \|_{\DD^1_\rho(D)}\le (R^D\rho,|\mu_c|)+(R^D\rho,\mathbf1_{\{|u|>n\}}|\mu_d|).
\end{equation}
From (\ref{eq5.2}) and (\ref{eq5.3}) one easily concludes the desired convergence.
The second assertion of the theorem is an immediate consequence of the first one.
\end{proof}

\section{Reconstruction formula via carr\'e du champ operator and jump measure}
\label{sec6}

Our aim in this section is to recover $\mu_c$ from the  energy of $u$.
As in Section \ref{sec5} we assume that (\ref{eq5.4}) is satisfied.
We  start with a useful lemma.

\begin{lemma}
\label{lem6.1}
For any $x,y\ge 0$ and $f\in C_b(\mathbb R)$ we have
\[
\int_0^\infty\Big[(x-a)^+-(y-a)^+-\mathbf1_{\{y>a\}}(x-y)\Big]f(a)\,da
=(x-y)^2{\sigma}(f;x,y),
\]
where
\[
\sigma(f;x,y)=\int_0^1\!\!\int_0^1\alpha f(\alpha\beta(x-y)+y)\,d\alpha\,d\beta.
\]
Furthermore, if $f_n\equiv \mathbf1_{[n,2n]}$, then
\[
(x-y)^2{\sigma}(f_n;x,y)= \frac12\big(S_n(x)-S_n(y)\big)\big(2x-S_n(x)-S_n(y)\big),\quad x,y\ge 0,
\]
where $S_n(z)= \max\{\min\{z,2n\},n\}$, $z\ge 0$.
\end{lemma}
\begin{proof}
The first part is just a simple application  of the  fundamental theorem of calculus
to the function $g(x):= \int_0^x(x-a)f(a)\,da$. The second part is a matter of straightforward computation.
\end{proof}

Let $\EE^{(c)}$ denote the strongly local part of the Beurling--Deny decomposition of $\EE$ (see \cite[Theorem 3.2.1]{FOT} or \cite[Theorem 4.3.3]{CF}). By
\cite[Exercise 4.3.12]{CF}, for any $w\in F(D)\cap \mathcal B_b(D)$ there exists a unique
nonnegative Radon measure $\mu_{\langle w\rangle}^{(c)}$ such that
\begin{equation}
\label{eq6.6}
\int_D \eta\,d\mu^{(c)}_{\langle w\rangle}=2\EE^{(c)}(w\eta,w)-\EE^{(c)}(w^2,\eta),\quad \eta\in C_b(D)\cap F(D).
\end{equation}

Suppose that $u$ solves  \eqref{eq1.5}. By the probabilistic definition of a solution of \eqref{eq1.5},
$u(X)$ is a special semimartingale under the measure $P^D_x$ for q.e. $x\in D$.
Let $\Gamma_c(u,u)$ be the Revuz measure of the positive continuous additive functional  $[u(X)]^c$  (the continuous part of the quadratic variation of $u(X)$) of $\BM^D$.
By \cite{KR:NoD}, $T_k(u)\in F(D)$. Consequently, by \cite[Lemma 3.2.3, Lemma 5.3.3]{FOT},
\begin{equation}
\label{eq6.1}
\mathbf{1}_{\{-k<u\le k\}}\Gamma_c(u,u)
=\mu^{(c)}_{\langle T_k(u)\rangle},\quad k\ge 1.
\end{equation}
Note that by \cite[Lemma 3.2.3]{FOT}, $\mathbf{1}_{\{-k<u\le k\}}\Gamma_c(u,u)$   is bounded, so $\Gamma_c(u,u)$
is $\sigma$-finite. Let $J$ and $\kappa$ be the jump measure and the killing measure, respectively, of the
Beurling--Deny decomposition of $\EE$.

For any nonnegative function $g\in \mathcal B(\BR\times\BR)$ we set
\begin{align*}
\Gamma^g_j(u,u)(dx)&= \textcolor{red}{4}\int_{D_\partial} |u(x)-u(y)|^2g(u(x),u(y))J(dx,dy)\\
&=\textcolor{red}{4}\int_{D} |u(x)-u(y)|^2g(u(x),u(y))J(dx,dy)+\textcolor{red}{4}|u(x)|^2g(u(x),0)\kappa_D(dx),
\end{align*}
where
\[
\kappa_D(dx)=\kappa(dx)+\mathbf 1_D\cdot J(dx,D^c).
\]

In the proof of the following result we shall frequently use the   identity
$\langle R^D\nu_1,\nu_2\rangle=\langle \nu_1,R^D\nu_2\rangle$, which  is a simple
consequence of symmetry of $\EE$.

\begin{theorem}
\label{th6.2}
Let $u$ be an    integral solution of \eqref{eq1.5}. Then for any $\eta\in C_c(D)$,
\begin{equation}
\label{eq6.2}
\begin{split}
\frac1{2n}\Big[\int_{\{n\le u\le 2n\}}\eta\,d\Gamma_c(u,u)&+ \int_D\int_D \eta(x)\theta_n(u(x),u(y))\,J(dx,dy)
\\&+\int_D \eta(x)\theta_n(u(x),0)\,\kappa_D(dx)\Big]\to \int_D \eta\,d\mu^+_c
\end{split}
\end{equation}
as $n\rightarrow\infty$, where
\[
\theta_n(u(x),u(y))= \textcolor{red}{2}\big(S_n(u(x))-S_n(u(y))\big)\big(2u(x)-S_n(u(x))-S_n(u(y))\big)
\]
with $S_n(z)= \max\{\min\{z,2n\},n\}$.
\end{theorem}
\begin{proof}
First note that by \cite[Proposition 3.7]{K:NoD}, without loss of generality, we may assume that $u$ is nonnegative and  $\mu_c=\mu^+_c$. To shorten notation, for $f\in\mathcal B(E)$ we write $\sigma_f(\cdot,\cdot)= \sigma(f;u(\cdot),u(\cdot))$
(see Lemma \ref{lem6.1}). Let $(\tau_k)$ be a reducing sequence for $u$. By the Tanaka--Meyer formula (see \cite[Theorem IV.7.70]{P}), for any $a\ge 0$ we have
\begin{align}
\label{eq6.3}
(u-a)^+(x)&=\mathbb E^D_x(u-a)^+(X_{\tau_k})+\mathbb E^D_x\int_0^{\tau_k}\mathbf1_{\{u(X_{s-})>a\}}\,dA^{\mu_d}_s
-\frac12L^a_{\tau_k}\nonumber\\
&\quad - \sum_{0\le s\le \tau_k} \Big( (u(X_s)-a)^+-(u(X_{s-})-a)^+
-\mathbf1_{\{u(X_{s-})>a\}}\Delta u(X_s)\Big),
\end{align}
where $L^a$ is the local time of $u(X)$ at $a$. Suppose that  $f\in C_c(\mathbb R^+)$. Then, by Lemma \ref{lem6.1},
\[
\begin{split}
\int_0^\infty\Big(\sum_{0\le s\le \tau_k} \Big( (u(X_s)-a)^+&-(u(X_{s-})-a)^+-\mathbf1_{\{u(X_{s-})>a\}}\Delta u(X_s)\Big)f(a)\,da\\
&=\sum_{0\le s\le \tau_k} |\Delta u(X_s)|^2\sigma(f;u(X_s),u(X_{s-})).
\end{split}
\]
Let $(N,H)$ be a L\'evy system of the process $\BM^D$ (see, e.g., \cite{Cinlar}
or \cite[Section A.3.4]{CF}). Then (see \cite[(A.3.33]{CF})
\begin{align}
\label{eq6.4}
&\BE^D_x \sum_{0\le s\le \tau_k} |\Delta u(X_s)|^2\sigma(f;u(X_s),u(X_{s-}))\nonumber\\
&\qquad =\BE^D_x \int_0^{\tau_k}\int_{D_\partial}
(u(z)-u(X_s))^2 \sigma(f;u(z),u(X_s))\,N(X_s,dz)\,dH_s.
\end{align}
Furthermore, by \cite[Corollary 1 to Theorem IV.70]{P},
\begin{equation}
\label{eq6.5}
\int_{\mathbb R}f(a) L^a_{\tau_k}\,da=\int_0^{\tau_k}f(u(X_s))\,dA^{\Gamma_c(u,u)}_s.
\end{equation}
From (\ref{eq6.3})--(\ref{eq6.5}) we get
\begin{align*}
\int_0^{u(x)}f(a)\,da &=\int_{\BR}(u(x)-a)^+f(a)\,da\\
&=\int_{\BR}\BE^D_x(u-a)^+(X_{\tau_k})f(a)\,da\\
&\quad+\int_{\BR}\BE^D_x\int_0^{\tau_k}\mathbf1_{\{u(X_{s-})>a\}}f(a)\,dA^{\mu_d}_s\,da
-\frac12\BE_x^D\int_0^{\tau_k}f(\textcolor{red}{u}(X_s))\,dA^{\Gamma_c(u,u)}_s\\
&\quad- \BE^D_x \int_0^{\tau_k}\!\!\int_{D_\partial} (u(z)-u(X_s))^2 \sigma(f;u(z),u(X_s))\,N(X_s,dz)\,dH_s.
\end{align*}
Hence, for any $\nu\in \mathcal P$,
\begin{align}
\label{eq6.7}
&\frac12\BE_\nu^D\int_0^{\tau_D}f(u(X_s))\,dA^{\Gamma_c(u,u)}_s\nonumber \\
&\qquad+\BE^D_\nu \int_0^{\tau_D}\!\!\int_{D_\partial}
(u(z)-u(X_s))^2 \sigma(f;u(z),u(X_s))\,N(X_s,dz)\,dH_s \nonumber\\
&\qquad=-\int_0^{\langle u,\nu\rangle}f(a)\,da +I(f)
\end{align}
with
\[
I(f)=\sup_{V\in\mathcal O_q(D)}\Big[\int_{\BR} f(a)\Big(\mathbb E^D_\nu(u-a)^+(X_{\tau_V})+
\BE^D_\nu\int_0^{\tau_V}\mathbf1_{\{u(X_{s-})>a\}}\,dA^{\mu_d}_s\Big)\,da\Big].
\]
By \cite[Theorem 5.3.1]{FOT},
$2J(dx\,dy)=N(x,dy)\mu_H(dx)$, where $\mu_H$ is the Revuz measure of $H$. Therefore from
(\ref{eq6.7}) it follows that if  $\mbox{supp}[f]\subset [\langle u,\nu\rangle,\infty)$, then
\[
\frac12\left[\langle R^D\nu, f(u)\cdot\Gamma_c(u,u)\rangle
+ \langle R^D\nu, \Gamma^{\sigma_f}_j(u,u)\rangle\right]= I(f).
\]
In particular, taking  $f_n:=\fch_{[n,2n]}$ with a sufficiently  large $n$ we get
\begin{align*}
&\frac12\left[\langle R^D\nu, \mathbf1_{\{n\le u\le 2n\}}\cdot\Gamma_c(u,u)\rangle
+ \langle R^D\nu, \Gamma^{\sigma_{f_n}}_j(u,u)\rangle\right]\\
&\qquad=\sup_{V\in\mathcal O_q(D)}\Big[\int_n^{2n} \Big(\mathbb E^D_\nu(u-a)^+(X_{\tau_V})+
\mathbb E^D_\nu\int_0^{\tau_V}\mathbf1_{\{u(X_{s-})>a\}}\,dA^{\mu_d}_s\Big)\,da\Big].
\end{align*}
Observe that
\[
\frac1n\sup_{V\in\mathcal O_q(D)}\Big[\int_n^{2n}
\Big(\BE^D_\nu\int_0^{\tau_V}\mathbf1_{\{u(X_{s-})>a\}}\,dA^{\mu_d}_s\Big)\,da\Big]\le
\BE^D_\nu\int_0^{\tau_D}\mathbf1_{\{u(X_{s-})>n\}}\,dA^{\mu_d}_s\to 0
\]
and
\begin{align*}
\sup_{V\in\mathcal O_q(D)}\mathbb E^D_\nu(u-2n)^+(X_{\tau_V})
&\le\frac{1}{n}\sup_{V\in\mathcal O_q(D)}\Big[\int_n^{2n} \mathbb E^D_\nu(u-a)^+(X_{\tau_V})\,da\Big]\\
&\le \sup_{V\in\mathcal O_q(D)}\mathbb E^D_\nu(u-n)^+(X_{\tau_V}).
\end{align*}
As a result, by Theorem \ref{th5.8},
\[
\frac1{2n}\langle R^D\nu, \fch_{\{n\le u\le 2n\}}\cdot\Gamma_c(u,u)
+\Gamma^{\sigma_{f_n}}_j(u,u)\rangle\to \int_D R^D\nu\,d\mu_c
\]
as $n\rightarrow\infty$. From this and Lemma \ref{lem6.1} we  easily get (\ref{eq6.2}).
\end{proof}

\begin{remark}
Observe that $\theta_n$ of Theorem \ref{th6.2} equals $2|u(x)-u(y)|^2$
when $u(x),u(y)\in [n,2n]$
and equals zero when $u(x), u(y)\le n$ or $u(x),u(y)\ge 2n$.
\end{remark}

\section{Examples and additional remarks}
\label{sec7}

In the following examples $D$ is a nonempty open bounded set in $E:=\BR^d$, $d\ge3$, and $m$ is the Lebesgue measure. Boundedness of $D$ implies in particular that in all the examples given below condition (\ref{eq5.4}) is satisfied.

\begin{example}
\label{ex4.1}
(Laplace operator). Consider the form
\[
\EE(u,v)=\frac12\int_{\BR^d}\nabla u\cdot\nabla v\,dx,\quad u,v\in \mathfrak D(\EE):=H^1(\BR^d).
\]
It is known (see \cite[Examples 1.2.3, 1.5.1]{FOT}) that $(\EE,H^1(\BR^d))$ is a transient regular Dirichlet  form on $L^2(\BR^d;m)$. The operator associated with $\EE$ is $(1/2)\Delta$. The process $\BM$ associated with it in the resolvent sense is a standard $d$-dimensional Brownian motion (see \cite[Example 4.2.1]{FOT}). It is known (see \cite[Exercise 2.3.1]{FOT} that $\mathfrak D(\EE^D)=H^1_0(D)$, so the part $\BM^D$ of $\BM$ on $D$ is nothing but the process associated with the form $(\EE^D,H^1_0(D))$. Since $P^D_t(x,\cdot)\le P_t(x,\cdot)$,
the process $\BM^D$ satisfies (\ref{eq2.6}).
\end{example}

We say that $u\in C^2(\bar D)$ if there exists $U\in C^2(\BR^d)$ such that $U=u$ in $\bar D$. Set $C^2_0(\bar D)=\{u\in C^2(\bar D):u=0$ on $\partial D\}$.

\begin{remark}
Following  \cite[Definition (5.1)]{LSW} (see also \cite[D\'efinition 9.1]{S} and \cite{MV}) we say that $u\in L^1(D;m)$ is  a {\em weak solution} of (\ref{eq1.1}) if
\[
-\frac12\int_Du\Delta v\,dx=\int_Dv\,d\mu,\quad
v\in C^2_0(\bar D).
\]
Assume additionally that $D$ is regular, say of class $C^2$. If $\mu\in\MMr_b(D)$ then, by \cite[Theorem 1.2.2]{MV}, problem (\ref{eq1.1}) has a unique weak solution $u$ given by (\ref{eq1.2}), i.e. the unique weak solution coincides with the integral solution. For further remarks see Remark \ref{rem7.6}.
\end{remark}

\begin{example}
(Divergence form operator). Let $a_{ij}:\BR^d\rightarrow\BR$ be measurable functions such that $a_{ij}(x)=a_{ji}(x)$ for $x\in\BR^d$ and $i,j=1,\dots,d$, and for some
strictly positive function $\lambda:\mathbb R^d\to\mathbb R$ and constant $\Lambda>0$ we have
\[
\lambda(x)|\xi|^2\le\sum^{d}_{i,j=1}a_{ij}(x)\xi_i\xi_j\le\Lambda|\xi|^2,\quad x,\xi=(\xi_1,\dots,\xi_d)\in\BR^d.
\]
Then the form
\[
\EE(u,v)=\sum^d_{i,j=1}\int_{\BR^d}a_{ij}(x)\partial_{x_i}u(x)\partial_{x_j}v(x)\,dx,
\quad u,v\in \mathfrak D(\EE):=H^1(\BR^d),
\]
is a regular  Dirichlet form on $L^2(E;m)$ satisfying  the absolute continuity condition provided that $\lambda^{-1}\in L^1_{loc}(\mathbb R^d)$
(see \cite{CMM,RW,VV}). The operator associated with $\EE$ has the form
\[
Lu=\sum^{d}_{i,j=1}\partial_{x_i}(a_{ij}(x)\partial_{x_j}u), \quad u\in \mathfrak D(L).
\]
Clearly, we have  $\EE^{(c)}=\EE$. From (\ref{eq6.6}) (and direct computation of its right-hand side) it follows that
\[
\mu^{(c)}_{\langle u\rangle}= \Gamma_c(u,u)  =2\sum^d_{i,j=1}a_{ij}(x)\partial_{x_i}u(x)\partial_{x_j}u(x)\,dx
=2(a\nabla u\cdot\nabla u)(x)\,dx.
\]
Alternatively, one can use the known formula for the additive functional $[u(X)]$; see \cite[Example 5.2.1]{FOT}. Therefore the reconstruction formula (\ref{eq6.2}) reads:
\[
\frac1n\int_{\{n\le u\le 2n\}} (a\nabla u\cdot\nabla u)(x)\eta(x)\,dx\rightarrow\int_D\eta\,d\mu^+_c.
\]
\end{example}

\begin{example}
(Fractional Laplace operator). Let $\alpha\in(0,2)$ and
\[
\EE(u,v)=\int_{\BR^d}\hat u(x)\bar{\hat v}(x)|x|^{\alpha}\,dx,\quad u,v\in \mathfrak D(\EE):=H^{\alpha/2}(\BR^d),
\]
where $\hat u$ denotes the Fourier transform of $u$ and
\[
H^{\alpha/2}(\BR^d)=\{u\in L^2(\BR^d):\int_{\BR^d}|\hat u(x)|^2|x|^{\alpha}\,dx<\infty\}.
\]
By \cite[Examples 1.4.1, 1.5.2]{FOT}, $(\EE,\mathfrak D(\EE))$ is a transient Dirichlet form on $L^2(\BR^d;m)$. Its generator is that fractional Laplace operator $-(-\Delta)^{\alpha/2}$. The Hunt process associated with $\EE$ is called a symmetric $\alpha$-stable L\'evy process
(see \cite[Example 4.1.1]{FOT}), and $\BM^D$ is the $\alpha$-symmetric stable L\'evy process killed upon leaving $D$. Condition (\ref{eq2.6}) is satisfied, because it is satisfied by the transition kernel of $\BM$.
In the Beurling--Deny decomposition of $\EE$ we have $\EE^{(c)}=0$, $\kappa=0$ and
\[
J(dx\,dy)=c(\alpha,d)|x-y|^{-d-\alpha}\,dx\,dy.
\]
Therefore  (\ref{eq6.2}) reads:
\[
\begin{split}
\frac{c(\alpha,d)}{n}&\int_D\int_D\eta(x)\big(S_n(u)(x)-S_n(u)(y)\big)
\big(2u(x)-S_n(u)(x)-S_n(u)(y)\big)|x-y|^{-d-\alpha}\,dx\,dy
\\&
+\frac{1}{n}\int_D\eta(x)\big(S_n(u)(x)-n\big)\big(2u(x)-S_n(u)(x)-n\big)\,\kappa_D(dx)
 \rightarrow \int_D\eta\,d\mu^+_c
\end{split}
\]
with
\[
\kappa_D(dx)=c(\alpha,d)\Big[ \int_{D^c} |x-y|^{-d-\alpha}\,dy\Big]\,dx.
\]
\end{example}

\begin{remark}
By \cite[Exercise 4.2.1]{FOT}, if $(P_t)_{t>0}$ is strongly Feller, that is $P_t(\BB_b(E))\subset C_b(E))$ for $t>0$, then $\BM$ satisfies (\ref{eq2.6}).
\end{remark}

\begin{remark}
\label{rem7.6}
Let $\mu\in\MMr_b(D)$.
By \cite[Proposition 4.12]{K:CVPDE}, the
integral solution of (\ref{eq1.5}) coincides with the unique duality solution of
(\ref{eq1.5}), and by
\cite[Theorem 4.4]{KR:MM} (see also \cite[Corollary 4.10]{K:NoD}) the integral solution of (\ref{eq1.5}) coincides with the unique renormalized solution of (\ref{eq1.5}).
Therefore in Theorem \ref{th5.8}  ``integral solution" can be replaced by ``duality solution" or by ``renormalized solution".
\end{remark}

\section*{Acknowledgments}
T. Klimsiak acknowledges the support of the NCN grant   2022/45/B/ST1/01095.


\begin{thebibliography}{20}

\bibitem{BGO}
L. Boccardo, T. Gallou\"et, L. Orsina, Existence and uniqueness
of entropy solutions for nonlinear elliptic equations with measure
data, {\em Ann. Inst. H. Poincar\'e Anal. Non Lin\'eaire} {\bf 13}
(1996) 539–-551.

\bibitem{Bil}
P. Billingsley,  {\em Convergence of Probability Measures}, Wiley,
New York 1968.

\bibitem{BG}
R. M. Blumenthal, R. K. Getoor, {\em Markov processes and potential
theory}, Academic Press, New York and London 1968.

\bibitem{CF}
Z.-Q. Chen, M. Fukushima, {\em Symmetric Markov processes, time
change, and boundary theory}. Princeton University Press,
Princeton, NJ, 2012.

\bibitem{Cinlar}
E. Cinlar,  Levy Systems for Markov Additive Processes, {\em Z. Wahrsch.
Verw. Gebiete} {\bf 31} (1975) 175--185.

\bibitem{CMM}
 C.V. Coffman, M.M. Marcus, V.J.  Mizel, On Green's function and eigenvalues of nonuniformly elliptic boundary value problems,
{\em Math. Z.} {\bf 182} (1983) 321--326.

\bibitem{DMOP}
G. Dal Maso, F. Murat, L.  Orsina, A. Prignet,
Renormalized solutions of elliptic equations with general
measure data, {\em Ann. Scuola Norm. Sup. Pisa Cl. Sci.} {\bf 28} (1999) 741--808.

\bibitem{DM}
C. Dellacherie, P.-A. Meyer, {\em Probabilities and Potential},
North-Holland Publishing Co., Amsterdam, 1978.

\bibitem{EK}
N. El Karoui,  Les aspects probabilites du controle stochastique,
{\em Lecture Notes in Math.} {\bf 876} (1981) 73--238.

\bibitem{FOT}
M. Fukushima, Y. Oshima, M. Takeda, {\em Dirichlet forms and
symmetric Markov processes.  Second revised and extended edition}.
Walter de Gruyter, Berlin, 2011.

\bibitem{FST}
M. Fukushima, K. Sato, S. Taniguchi, On the closable parts of
pre-Dirichlet forms and the fine supports of underlying measures,
{\em Osaka J. Math.} {\bf 28} (1991) 517--535.

\bibitem{K:CVPDE}
T. Klimsiak, Reduced measures for semilinear elliptic equations
involving Dirichlet operators,  {\em Calc. Var. Partial
Differential Equations} {\bf 55} (2016), Art. 78, 27 pp.

\bibitem{K:NoD}
T. Klimsiak, On uniqueness and structure of renormalized solutions
to integro-differential equations with general measure data,
{\em NoDEA Nonlinear Differential Equations Appl.} {\bf 27} (2020), Art. 47, 24 pp.

\bibitem{K:SPA}
T. Klimsiak,  Non-semimartingale solutions of reflected BSDEs and applications to Dynkin games,
{\em Stochastic Process. Appl.} {\bf 134} (2021) 208--239

\bibitem{KR:CM}
T. Klimsiak, A. Rozkosz,  Semilinear elliptic equations with
measure data and quasi-regular Dirichlet forms, {\em Colloq.
Math.} {\bf 145} (2016) 35--67.

\bibitem{KR:BPAN}
T. Klimsiak, A. Rozkosz, On the structure of bounded smooth measures associated
with a quasi-regular Dirichlet form,  {\em Bull. Pol. Acad. Sci. Math.} {\bf 65}
(2017) 45--56.

\bibitem{KR:NoD}
T. Klimsiak, A. Rozkosz,  On semilinear elliptic equations with diffuse measures,
{\em NoDEA Nonlinear Differential Equations Appl.} {\bf 25} (2018),  Paper No. 35, 23 pp.

\bibitem{KR:MM}
T. Klimsiak, A. Rozkosz, Renormalized solutions of semilinear elliptic
equations with general measure data, {\em  Monatsh. Math.} {\bf 188} (2019) 689--702.

\bibitem{KR:ALEA}
T. Klimsiak, A. Rozkosz,  Long-time asymptotic behavior of the value function in nonlinear stopping problems,
{\em ALEA, Lat. Am. J. Probab. Math. Stat.} {\bf 19} (2022) 1133--1160.

\bibitem{KRS}
T. Klimsiak, M.  Rzymowski, L. S\l omi\'nski,  Reflected BSDEs with two optional
barriers, {\em Bull. Sci. Math.} {\bf 158} (2020), 102820, 49 pp.

\bibitem{LeJan}
Y. Le Jan,  Quasi-continuous functions and Hunt processes,
{\em J. Math. Soc. Japan} {\bf 35} (1983) 37--42.

\bibitem{LSW}
W. Littman, G. Stampacchia, H.F. Weinberger,  Regular points for
elliptic equations with discontinuous coefficients, {\em Ann. Scuola
Norm. Sup. Pisa} {\bf 17} (1963) 43--77.

\bibitem{MV}
M. Marcus, L. V\'eron,  {\em Nonlinear second order elliptic equations
involving measures}. De Gruyter, Berlin, 2014.

\bibitem{P}
P. Protter, {\em Stochastic Integration and Differential
Equations}. Second Edition. Springer, Berlin, 2004.

\bibitem{RW}
M. R\"ockner,  N. Wielens,  Dirichlet forms--closability and change of speed measure,
in: Infinite-dimensional analysis and stochastic processes (Bielefeld, 1983),
pp. 119--144, Res. Notes in Math., 124, Pitman, Boston, 1985.

\bibitem{Sil}
M.L. Silverstein,  {\em Symmetric Markov processes}. Lecture Notes in Math. 426,
Springer-Verlag, Berlin-Heidelberg-New York, 1974.

\bibitem{S}
G. Stampacchia, Le probl\`eme de Dirichlet pour les \'equations
elliptiques du second ordre \`a coefficients discontinus, {\em
Ann. Inst. Fourier} {\bf 15} (1965) 189--258.

\bibitem{VV}
H. Vogt, J.  Voigt,  Wentzell boundary conditions in the context of Dirichlet forms,
{\em Adv. Differential Equations} 8 (2003) 821--842.

\end{thebibliography}
\end{document}